\documentclass[10pt]{amsart}
\usepackage{amssymb}

\newtheorem{proposition}{Proposition}[section]
\newtheorem{lemma}[proposition]{Lemma}
\newtheorem{corollary}[proposition]{Corollary}
\newtheorem{theorem}[proposition]{Theorem}
\newtheorem{remark}[proposition]{Remark}

\theoremstyle{definition}

\newcommand{\selabel}[1]{\label{se:#1}}

\def\<{\leq}
\def\>{\geq}
\def\a{\alpha}
\def\b{\beta}

\def\O{\Omega}

\def\ol{\overline}
\def\t{\triangle}
\def\e{\varepsilon}
\def\oo{\infty}

\def\s{\sigma}

\def\ot{\otimes}
\def\ra{\rightarrow}

\date{}

\begin{document}
\title{The Green Ring of Drinfeld Double $D(H_4)$}
\author{Hui-Xiang Chen}
\address{School of Mathematical Science, Yangzhou University,
Yangzhou 225002, China}
\email{hxchen@yzu.edu.cn}
%\subjclass{16E05, 16G99, 16T99}
\begin{abstract}
In this paper, we study the Green ring (or the representation ring) of Drinfeld quantum double $D(H_4)$ of Sweedler's
4-dimensional Hopf algebra $H_4$. We first give the decompositions of the tensor products of finite
dimensional indecomposable modules into the direct sum of indecomposable modules over $D(H_4)$.
Then we describe the structure of the Green ring $r(D(H_4))$ of $D(H_4)$ and show that $r(D(H_4))$ is generated,
as a ring, by infinitely many elements
subject to a family of relations.
\end{abstract}
\maketitle

\section*{\bf Introduction}
The tensor product of modules over a Hopf algebra is an important ingredient
in the representation theory of Hopf algebras and quantum groups.
In particular, the decomposition of the tensor product
of indecomposable modules into a direct sum of indecomposable modules has received enormous
attention. For modules over a finite dimensional group algebra,
this information is encoded in the
structure of the Green ring (or the representation ring), see
\cite{Archer, BenCar, BenPar, BrJoh, Green, HTW}).
For modules over a Hopf algebra or a quantum group there are results
on a quiver quantum group by Cibils \cite{Cib},
on  the quantum double of a finite group by Witherspoon
\cite{With}, on the half quantum groups
(or Taft algebras) by Gunnlaugsd$\acute{\rm o}$ttir \cite{Gunn}, on
the coordinate Hopf algebra of quantum ${\rm SL}(2)$ at a root of unity by Chin
\cite{Chin}. Kondo and Saito gave the indecomposable decomposition of tensor products of modules
over the restricted quantum universal enveloping algebra associated to $\mathfrak{sl}_2$
in \cite{ks}. However, the Green rings of those Hopf algebras are either equal to
the Grothendick rings (in the semisimple cases) or not yet computed because of the complexity.
Recently, Chen, Van Oystaeyen and Zhang computed the Green rings of Taft algebras $H_n(q)$ in \cite{ChVOZh},
Li and Zhang studied the Green rings of the generalized Taft algebras in \cite{LiZhang}.
Since the Taft algebras are of finite representation type, their Green rings are
finitely generated as rings. It was shown that the Green rings of Taft algebras
generated by two elements subject to certain relations for each $n\>2$ in \cite{ChVOZh}.
However, the Drinfeld quantum doubles $D(H_n(q))$ of Taft algebras
$H_n(q)$ are of infinite representation type \cite{Ch4}.
Hence the Green rings of the Drinfeld quantum doubles of $H_n(q)$ are much more complicated.
When $n=2$, the Taft algebra $H_2(q)$ is exactly the Sweedler's 4-dimensional Hopf algebra $H_4$
(see \cite{Sw, Ta}).
In this paper, we will investigate the Green ring of the Drinfeld quantum double $D(H_4)$.

The paper is organized as follow. In Section 1, we recall the definitions of Grothendieck ring and Green ring
(or representation ring) of a Hopf algebra, the structure of the Drinfeld quantum double $D(H_4)$ of Sweedler's
4-dimensional Hopf algebra $H_4$ and the finite dimensional indecomposable modules over $D(H_4)$.
In Section 2, we investigate the tensor products of finite dimensional
indecomposable modules over $D(H_4)$. We decompose the tensor product of any two indecomposable
 $D(H_4)$-modules into a direct sum of indecomposable modules.
In Section 3, we study the structure of the Green ring $r(D(H_4))$ of $D(H_4)$. We first investigate a subring $R$
of $r(D(H_4))$, which is generated, as a $\mathbb Z$-module, by the isomorphism classes of
the indecomposable modules located in the connected components
of the AR-quiver of $D(H_4)$ containing simple modules (or indecomposable projective modules).
We show that $R$ is generated, as a ring, by four elements subject to certain relations.
Then we investigate the structure of the Green ring $r(D(H_4))$. We give a family of generators of $r(D(H_4))$
and the relations satisfied by the generators, as a ring, which shows that $r(D(H_4))$ is not finitely generated
as a ring.

\section{\bf Preliminaries}\selabel{1}
Throughout, we work over an algebraically closed field $k$ with char$(k)\not=2$. Unless
otherwise stated, all algebras, Hopf algebras and modules are
defined over $k$; all modules are left modules and finite dimensional;
all maps are $k$-linear; dim, $\otimes$ and Hom stand for ${\rm dim}_k$, $\otimes_k$ and
Hom$_k$, respectively. For the theory of Hopf algebras and quantum groups, we
refer to \cite{Ka, Maj, Mon, Sw}. For the representation theory of finite dimensional algebras,
we refer to \cite{ARS}.
Let $\mathbb Z$ denote all integers, and ${\mathbb Z}_2={\mathbb Z}/2{\mathbb Z}$.

\subsection{Grothendieck rings and Green rings}\selabel{1.1}

For a finite dimensional algebra $A$, let mod$A$ denote the category of finite dimensional $A$-modules.
For a module $M$ in mod$A$ and a nonnegative integer $s$, let $sM$ denote the
direct sum of $s$ copies of $M$. Then $sM=0$ if $s=0$.
Let $P(M)$ denote the projective cover of $M$, and let $I(M)$
denote the injective envelope of $M$. Let $l(M)$ denote the length of $M$,
and let ${\rm rl}(M)$ denote the Loewy length (=radical length=socle length)
of $M$.

For a finite dimensional algebra $A$,
let $G_0(A)$ denote the {\it Grothendieck group} of the category mod$A$.
This is the abelian group that is generated by the isomorphism classes $[M]$ of
$A$-modules $M$ modulo the relations $[M]=[U]+[V]$ for each short exact sequence of
$0\ra U\ra M\ra V\ra 0$ in mod$A$.
It is well known (see \cite{ARS, Bass}) that $G_0(A)$ is a free abelian group with a $\mathbb Z$-basis given
by the classes [$S_i$], $i=1,2,\cdots, t$, where $\{S_1, S_2, \cdots, S_t\}$ is a full set
of non-isomorphic simple $A$-modules.

Let $H$ be a finite dimensional Hopf algebra. Then mod$H$ is a monoidal category \cite{Ka, Mon}.
Hence $G_0(H)$ is an associative ring with the
multiplication given by $[M][N]=[M\ot N]$ for any modules $M$ and $N$ in mod$H$.
The multiplication identity of $G_0(H)$ is $[k]$,
where $k$ is the trivial $H$-module given by the counit of $H$.
In this case, $G_0(H)$ is called the {\it Grothendieck ring} of $H$ (or of the
monoidal category mod$H$).

Let $H$ be a finite dimensional Hopf algebra.
The {\it representation rings} $r(H)$ and $R(H)$
can be defined as follows. $r(H)$ is the abelian group that is generated by the
isomorphism classes $[M]$ of $H$-modules $M$
modulo the relations $[M\oplus N]=[M]+[N]$ for any modules $M$ and $N$ in mod$H$.
The multiplication of $r(H)$
is given by the tensor product of $H$-modules, that is,
$[M][N]=[M\ot N]$. Then $r(H)$ is an associative ring with the identity $[k]$.
$R(H)$ is an associative $k$-algebra defined by $k\ot_{\mathbb Z}r(H)$.
Note that $r(H)$ is a free abelian group with a $\mathbb Z$-basis
$\{[M]|M\in{\rm ind}(H)\}$, where ${\rm ind}(H)$ denotes the category
of finite dimensional indecomposable $H$-modules.
$r(H)$ (resp. $R(H)$) is also called the {\it Green ring} of $H$
(or of the monoidal category mod$H$).
Note that there is canonical ring epimorphism
$r(H)\ra G_0(H)$, $[M]\mapsto [M]$, $M\in {\rm mod}H$.

If $H$ is a quasitriangular Hopf algebra, then $M\ot N\cong N\ot M$ for any $H$-modules
$M$ and $N$. In this case, both $G_0(H)$ and $r(H)$ are commutative rings.

A finite dimensional Hopf algebra $H$ is a symmetric algebra if and only if
$H$ is unimodular and $S^2$ is inner, where $S$ is the antipode of $H$ (see \cite{Lo, ObSch}).
It is well known \cite{Ra} that the Drinfeld double $D(H)$ of a finite dimensional Hopf algebra
is unimodular, and $S_{D(H)}^2$ is inner. Hence $D(H)$ is always symmetric.

Let $H$ be a finite dimensional Hopf algebra. For any module $M$ in mod$H$, the dual space $M^*={\rm Hom}(M, k)$
is also an $H$-module with the action given by
$$(h\cdot f)(m)=f(S(h)\cdot m), \ h\in H,\ f\in M^*,\ m\in M,$$
where $S$ is the antipode of $H$. It is well known that
$(M\ot N)^*\cong N^*\ot M^*$ for any $H$-modules $M$ and $N$.
If $H$ is quasitriangular, then $S^2$ is inner, and so $M^{**}\cong M$ for any
$M\in{\rm mod}H$ (see \cite{Lo}). In this case, this gives rise to a duality $(-)^*$ from ${\rm mod}H$ to itself,
which induces a ring involution of $r(H)$ (resp. $G_0(H)$) given by
$[M]^*=[M^*]$ for any $M\in{\rm mod}H$.

\subsection{Drinfeld double of $H_4$}\selabel{1.2}

Sweedler's 4-dimensional Hopf algebra is a special case of Taft Hopf algebras.
The Drinfeld quantum doubles of Taft's Hopf algebras and their finite representations
were investigated in \cite{Ch1, Ch2, Ch3, Ch4}. The representations of pointed Hopf algebras
and their Drinfeld quantum doubles were also studied in \cite{KropRad}.
Let us recall some results which we need throughout the paper.

Sweedler's $4$-dimensional Hopf algebra $H_4$ is
generated by two elements $g$ and $h$ subject to the relations:
$$g^2=1,\quad\quad h^2=0,\quad\quad gh+hg=0.$$
The coalgebra structure and the antipode are determined by
$$
\begin{array}{lll}
\t(g)=g\otimes g, & \t(h)=h\otimes g+1\otimes h, & \e(x)=0,\\
\e(h)=0, & S(g)=g^{-1}=g, &  S(h)=gh.
\end{array}
$$
Moreover, $H_4$ has a canonical basis $\{1$, $g$, $h$, $gh\}$.\\

Let $D_4$ be the algebra generated by
$a,b,c$ and $d$ subject to the relations:
$$\begin{array}{lllll}
ba=-ab,& db=-bd, & ca=-ac,& dc=-cd,& bc=cb,\\
a^2=0, & b^2=1, &c^2=1,& d^2=0, & da+ad=1-bc.
\end{array}$$
Then $D_4$ is a Hopf algebra with the coalgebra structure and the antipode given by
$$
\begin{array}{lll}
\t(a)=a\otimes b+1\otimes a, & \e(a)=0, & S(a)=-ab=ba,\\
\t(b)=b\otimes b, & \e(b)=1, & S(b)=b^{-1}=b,\\
\t(c)=c\otimes c,& \e(c)=1, & S(c)=c^{-1}=c,\\
\t(d)=d\otimes c+1\otimes d,& \e(d)=0, & S(d)=-dc=cd.
\end{array}
$$
$D_4$ is a $2^4$-dimensional Hopf algebra.
$D_4$ has a canonical basis $\{a^ib^jc^ld^k|0\leq
i,j,l,k\le 1\}$, and is not semisimple and is isomorphic to $D(H_4)$ as a Hopf algebra.
The Hopf algebra isomorphism is given by
$$ D(H_4)=H_4^{*cop}\bowtie H_4\rightarrow D_4, \ \
\overline{h^sg^t}\bowtie h^ig^j\mapsto
\sum_{0\leq m<2}\frac{1}{2}(-1)^{tm} c^md^s a^ib^j$$
for all $0\leq s,t,i,j\le 1$, where $\{h^ig^j|0\leq i, j\le 1\}$ is the basis of $H_4$,
and $\{\overline{h^ig^j}|0\leq i, j\le 1\}$ is the dual basis of $H_4^*$.
The canonical quasitriangular structure on $D_4$ reads as follows:
$$\begin{array}{rcl}
{\mathcal R}&=&\frac{1}{2}(1\ot 1+b\ot 1+1\ot c-b\ot c\\
&&+a\ot d+ab\ot d+a\ot cd-ab\ot cd).\\
\end{array}$$
For the detail, the reader is directed to \cite{Ch1, Ch2, ChZh}.

\subsection{Indecomposable representations of $D_4$}

Let $J(D_4)$ stand for the Jacobson radical of $D_4$.
Then  $J(D_4)^3=0$ by \cite[Corollary 2.4]{Ch4}. This means that the Loewy length
of $D_4$ is 3. In order to study the Green ring of $D_4$,
we need first to give the structures of all finite dimensional indecomposable $D_4$-modules.
We will follow the notations of \cite{Ch4}.

From \cite{Ch4}, we know that the socle series and the radical series of an indecomposable
$D_4$-module coincide. We list all indecomposable $D_4$-modules according to the Loewy length.
There are four simple $D_4$-modules (up to isomorphism);
two are of dimension one and two are of dimension two.
In the following, denote $J(D_4)$ by $J$ for short.

One dimensional simple modules: $V(1,r)$, $r\in{\mathbb Z}_2$,
\begin{equation*}
a\cdot v=d\cdot v=0,\hspace{0.2cm}b\cdot v=c\cdot v=(-1)^rv,
\hspace{0.2cm}v\in V(1,r).
\end{equation*}
In the following, denote $V(1, r)$ by $V(r)$, $r\in{\mathbb Z}_2$.

Two dimensional simple modules: $V(2,r)$, $r\in{\mathbb Z}_2$. $V(2,r)$ has
a standard $k$-basis $\{v_1, v_2\}$ such that
\begin{equation*}
\begin{array}{llll}
a\cdot v_1=v_2,& d\cdot v_1=0, & b\cdot v_1=(-1)^rv_1, & c\cdot v_1=(-1)^{r+1}v_1,\\
a\cdot v_2=0 , & d\cdot v_2=2v_1,& b\cdot v_2=(-1)^{r+1}v_2, & c\cdot v_2=(-1)^rv_2.
\end{array}
\end{equation*}
The simple module $V(2,r)$, $r\in{\mathbb Z}_2$, are both  projective  injective.

Four dimensional projective modules of Loewy length 3: Let $P(r)$ be the  projective cover of
$V(r)$, $r\in{\mathbb Z}_2$. Then $P(r)$ is the injective envelope of $V(r)$ as well, $r\in{\mathbb Z}_2$.
$P(r)$ has a standard $k$-basis $\{v_1, v_2, v_3, v_4\}$ such that
\begin{equation*}
\begin{array}{llll}
a\cdot v_1=v_2,& d\cdot v_1=v_3,& b\cdot v_1=(-1)^rv_1, & c\cdot v_1=(-1)^rv_1,\\
a\cdot v_2=0,& d\cdot v_2=-v_4, & b\cdot v_2=(-1)^{r+1}v_2, & c\cdot v_2=(-1)^{r+1}v_2,\\
a\cdot v_3=v_4, & d\cdot v_3=0, & b\cdot v_3=(-1)^{r+1}v_3, &c\cdot v_3=(-1)^{r+1}v_3,\\
a\cdot v_4=0, & d\cdot v_4=0,   & b\cdot v_4=(-1)^rv_4, & c\cdot v_4=(-1)^rv_4.
\end{array}
\end{equation*}
Note that ${\rm soc}(P(r))=J^2P(r)\cong V(r)$,
${\rm soc}^2(P(r))/{\rm soc}(P(r))=(JP(r))/(J^2P(r))
\cong 2V(r+1)$ and
$P(r)/{\rm soc}^2(P(r))=P(r)/(JP(r))\cong V(r)$.
Note that the $P(r)$ is exactly the $P(1, r)$ in \cite{Ch4}.

There are infinitely many non-isomorphic indecomposable $D_4$-modules with Loewy length 2.
We list them according to the lengths and the co-lengths of their socles.
We say that an indecomposable $D_4$-module $M$ with ${\rm rl}(M)=2$
is of $(s,t)$-type if $l(M/{\rm soc}(M))=s$ and $l({\rm soc}(M))=t$. By \cite{Ch4},
if $M$ is of $(s, t)$-type, then $s=t+1$, or $s=t$, or $s=t-1$.

The indecomposable modules of $(s+1,s)$-type are given by the syzygy functor $\O$.
Let $V(r)$ be the one dimensional simple modules, $r\in{\mathbb Z}_2$.
Then the minimal projective
resolutions of $V(r)$ are given
by
$$
\cdots\ra4P(r+1)\ra3P(r)\ra2P(r+1)\ra P(r)\ra
V(r)\ra 0.$$
By these resolutions, one can describe the structure of $\O^sV(r)$, $s\>1$
(see \cite{Ch4}). $\O^sV(r)$ is of $(s+1, s)$-type.
%In particular,
%the syzygy $\O V(r)\cong {\rm Ker}(P(r)\twoheadrightarrow V(r))=J\cdot P(r)$
%is of $(2,1)$-type.

The indecomposable modules of $(s,s+1)$-type are given by the cosyzygy functor $\O^{-1}$.
Let $V(r)$ be the one dimensional simple modules, $r\in{\mathbb Z}_2$.
Then the minimal injective resolutions of $V(r)$ are given
by
$$
0\ra V(r)\ra P(r)\ra 2P(r+1)\ra 3P(r)\ra 4P(r+1)
\ra\cdots.$$
By these resolutions, one can describe the structure of $\O^{-s}V(r)$, $s\>1$
(see \cite{Ch4}). $\O^{-s}V(r)$ is of $(s, s+1)$-type.

Let $r\in{\mathbb Z}_2$ and $s\>1$. If $s$ is odd, then we have
$D_4$-module isomorphisms
$$\begin{array}{c}
{\rm soc}(\O^sV(r))\cong\O^{-s}V(r)/{\rm soc}(\O^{-s}V(r))
\cong sV(r),\\
\O^sV(r)/{\rm soc}(\O^sV(r))\cong {\rm soc}(\O^{-s}V(r))\cong(s+1)V(r+1).\\
\end{array}$$
If $s$ is even, then we have $D_4$-module isomorphisms
$$\begin{array}{c}
{\rm soc}(\O^sV(r))\cong\O^{-s}V(r)/{\rm soc}(\O^{-s}V(r))
\cong sV(r+1),\\
\O^sV(r)/{\rm soc}(\O^sV(r))\cong {\rm soc}(\O^{-s}V(r))\cong(s+1)V(r).\\
\end{array}$$

The indecomposable modules of $(s, s)$-type can be described as follows.
Let ${\mathbb P}^1(k)$ be the projective 1-space over $k$.
${\mathbb P}^1(k)$ can be regarded as the set of all 1-dimensional
subspaces of $k^2$. Let $\oo$ be a symbol with $\oo\not\in k$ and
let $\overline k=k\cup\{\oo\}$. Then there is a bijection between
$\overline k$ and $\mathbb{P}^1(k)$: $\a\mapsto L(\a,1)$,
$\oo\mapsto L(1,0)$, where $\a\in k$ and $L(\a,\b)$ denotes the
1-dimensional subspace of $k^2$ with basis $(\a,\b)$ for any
$0\not=(\a,\b)\in k^2$. In the following, we regard
$\mathbb{P}^1(k)=\overline k$.

If $M$ is of $(s,s)$-type then $M\cong M_s(1,r,\eta)$, where $r\in{\mathbb Z}_2$
and $\eta\in{\mathbb P}^1(k)$ (see \cite{Ch4}). Denote $M_s(1,r,\eta)$ by $M_s(r,\eta)$
in the following.

The indecomposable module $M_1(r,\oo)$, $r\in{\mathbb Z}_2$, has a standard basis $\{v_1, v_2\}$ with the $D_4$-action given by
\begin{equation*}
\begin{array}{llll}
a\cdot v_1=0, & d\cdot v_1=v_2, & b\cdot v_1=(-1)^{r+1}v_1, & c\cdot v_1=(-1)^{r+1}v_1,\\
a\cdot v_2=0, & d\cdot v_2=0, & b\cdot v_2=(-1)^{r}v_2,& c\cdot v_2=(-1)^{r}v_2.\\
\end{array}
\end{equation*}

The indecomposable module $M_1(r,\eta)$, $r\in{\mathbb Z}_2$, $\eta\in k$,
has a standard basis $\{v_1, v_2\}$ with the $D_4$-action given by
\begin{equation*}
\begin{array}{llll}
a\cdot v_1=v_2, & d\cdot v_1=-\eta v_2, & b\cdot v_1=(-1)^{r+1}v_1, & c\cdot v_1=(-1)^{r+1}v_1,\\
a\cdot v_2=0, & d\cdot v_2=0, & b\cdot v_2=(-1)^{r}v_2, & c\cdot v_2=(-1)^{r}v_2.\\
\end{array}
\end{equation*}

For any $r\in{\mathbb Z}_2$ and $\eta\in{\mathbb P}^1(k)$, there is a unique
$D_4$-module injection $M_1(r,\eta)\hookrightarrow P(r)$, up to a nonzero scale multiple.
Moreover, there is an exact sequence of $D_4$-modules
$$0\ra M_1(r,\eta)\hookrightarrow P(r)\ra M_1(r+1,\eta)\ra 0.$$
Hence $M_1(r,\eta)$ is a submodule of $P(1,r)$ and a quotient module of $P(r+1)$.
%Moreover, soc$(M_1(r,\eta))=J\cdot M_1(r,\eta)\cong V(r)$ and
%$M_1(r,\eta)/{\rm soc}(M_1(r,\eta))\cong V(r+1)$.

Then one can construct $M_s(r,\eta)$ recursively by using pullback, where
$r\in{\mathbb Z}_2$ and $\eta\in{\mathbb P}^1(k)$ (see \cite[pp. 2823-2824]{Ch4}).
$M_s(r,\eta)$ is a submodule of $sP(1,r)$ and a quotient module of $sP(r+1)$,
and there is an exact sequence of $D_4$-modules
$$0\ra M_s(r,\eta)\hookrightarrow sP(r)\ra M_s(r+1,\eta)\ra 0.$$
%Moreover, soc$(M_s(r,\eta))=J\cdot M_s(r,\eta)\cong sV(r)$ and
%$M_s(r,\eta)/{\rm soc}(M_s(r,\eta))\cong sV(r+1)$.
Hence $\O M_s(r+1,\eta)\cong \O^{-1} M_s(r+1,\eta)\cong M_s(r,\eta)$.
Moreover, for any $1\<i<s$, $M_s(r,\eta)$ contains a unique
submodule of $(i,i)$-type, which is isomorphic to $M_i(r,\eta)$
and the quotient module of $M_s(r,\a)$ modulo the submodule of
$(i,i)$-type is isomorphic to $M_{s-i}(r,\eta)$. Hence there is
an exact sequence of $D_4$-modules
$$0\ra M_i(r,\eta)\hookrightarrow M_s(r, \eta)\ra M_{s-i}(r,\eta)\ra 0.$$

\section{\bf The tensor products of indecomposable modules}

In this section, we investigate the tensor products of two indecomposable modules.
We will give the indecomposable decomposition of the tensor products
of indecomposable modules over $D_4$.
Note that $M\ot N\cong N\ot M$ for any $D_4$-modules $M$ and $N$
since $D_4$ is a quasitriangular Hopf algebra.

\begin{proposition}\label{2.1} Let $r, r'\in{\mathbb Z}_2$ and $\eta\in{\mathbb P}^1(k)$.
Then there are $D_4$-module isomorphisms
$$\begin{array}{rl}
V(r)\ot V(r')\cong V(r+r'),& V(r)\ot V(2, r')\cong V(2, r+r'),\\
V(r)\ot P(r')\cong P(r+r'),&V(r)\ot M_1(r',\eta)\cong M_1(r+r',\eta).\\
\end{array}$$
\end{proposition}

\begin{proof} It follows from a straightforward verification.
\end{proof}

\begin{lemma}\label{2.2}  Let $r\in{\mathbb Z}_2$, and let $M$ be a $D_4$-module.
Then $V(r)\ot M$ is indecomposable if and only if $M$ is indecomposable.
Moreover, if $M$ is indecomposable then ${\rm rl}(V(r)\ot M)={\rm rl}(M)$, and
if $M$ is of $(s, t)$-type then so is $V(r)\ot M$.
\end{lemma}

\begin{proof}
If $M=N\oplus L$ for some nonzero submodules $N$ and $L$ of $M$, then
$V(r)\ot N$ and $V(r)\ot L$ are nonzero submodules of $V(r)\ot M$, and
$V(r)\ot M=V(r)\ot(N\oplus L)=(V(r)\ot N)\oplus (V(r)\ot L)$.
Hence if $V(r)\ot M$ is indecomposable then so is $M$.
Conversely, since $M\cong V(0)\ot M\cong (V(r)\ot V(r))\ot M\cong V(r)\ot(V(r)\ot M)$
by Proposition \ref{2.1}, the same argument as above shows that if $M$ is indecomposable
then so is $V(r)\ot M$.

Now assume that $M$ is indecomposable. If ${\rm rl}(M)=1$ or $3$, then $M$ is simple or projective, and hence
${\rm rl}(V(r)\ot M)={\rm rl}(M)$ by Proposition \ref{2.1}.

If ${\rm rl}(M)=2$ and $M$ is of $(s, t)$-type,
then there is an $r'\in\mathbb Z$ such that soc$(M)=JM\cong tV(r')$ and
$M/{\rm soc}(M)\cong sV(r'+1)$.
Let $0\neq v\in V(r)$. Then $V(r)=kv$ by dim$V(r)=1$. Hence any subspace of $V(r)\ot M$
has the form $v\ot N$ for some subspace $N$ of $M$.
%For any $m\in M$, we have
%$a\cdot(v\ot m)=a\cdot v\ot b\cdot m+v\ot a\cdot m=v\ot a\cdot m$,
%$b\cdot(v\ot m)=b\cdot v\ot b\cdot m=(-1)^rv\ot b\cdot m$,
%$c\cdot(v\ot m)=c\cdot v\ot c\cdot m=(-1)^rv\ot c\cdot m$ and
%$d\cdot(v\ot m)=d\cdot v\ot c\cdot m+v\ot d\cdot m=v\ot d\cdot m$.
It follows from a straightforward verification
that $v\ot N$ is a (simple) submodule of $V(r)\ot M$ if and only if $N$ is a (simple) submodule
of $M$. Thus, by Proposition \ref{2.1} we have ${\rm soc}(V(r)\ot M)=V(r)\ot{\rm soc}(M)
\cong V(r)\ot(tV(r'))\cong tV(r+r')$ and
$(V(r)\ot M)/{\rm soc}(V(r)\ot M)=(V(r)\ot M)/(V(r)\ot{\rm soc}(M))\cong
V(r)\ot(M/{\rm soc}(M))\cong V(r)\ot(sV(r'+1))\cong sV(r+r'+1)$.
Hence ${\rm rl}(V(r)\ot M)=2$ and $V(r)\ot M$ is of $(s, t)$-type.
\end{proof}

\begin{corollary}\label{2.3}
Let $r, r'\in{\mathbb Z}_2$ and $s\>1$. Then there are $D_4$-module isomorphisms
$$V(r)\ot\O^sV(r')\cong\O^sV(r+r'),\ V(r)\ot\O^{-s}V(r')\cong\O^{-s}V(r+r').$$
\end{corollary}

\begin{proof}
It follows from Lemma \ref{2.2} and its proof.
\end{proof}

\begin{proposition}\label{2.4}
Let $r, r'\in{\mathbb Z}_2$ and $\eta\in{\mathbb P}^1(k)$. Then
$V(r)\ot M_s(r',\eta)\cong M_s(r+r',\eta)$ as $D_4$-modules for all
$s\>1$.
\end{proposition}

\begin{proof}
It follows from Proposition \ref{2.1} that $V(r)\ot M_1(r',\eta)\cong M_1(r+r',\eta)$.
Now assume $s>1$. Since $M_s(r', \eta)$ is of $(s, s)$-type and soc$(M_s(r', \eta))\cong sV(r')$,
it follows from Lemma \ref{2.2} and its proof that $V(r)\ot M_s(r',\eta)$ is indecomposable of
$(s, s)$-type, and soc$(V(r)\ot M_s(r',\eta))\cong sV(r+r')$.
Then from \cite[Proposition 3.11 and Theorem 3.10(1)]{Ch4}, one gets that
$V(r)\ot M_s(r',\eta)\cong M_s(r+r', \a)$ for some $\a\in{\mathbb P}^1(k)$.
Thus, by \cite[Theorem 3.10(2)]{Ch4}, we know that $V(r)\ot M_s(r',\eta)$ contains a unique
submodule of $(1, 1)$-type, which is isomorphic to $M_1(r+r', \a)$.
On the other hand, again by \cite[Theorem 3.10(2)]{Ch4},
$M_s(r', \eta)$ contains a submodule isomorphic to $M_1(r', \eta)$.
Hence $V(r)\ot M_s(r',\eta)$ contains a submodule isomorphic to $M_1(r+r', \eta)$
since $V(r)\ot M_1(r', \eta)\cong M_1(r+r', \eta)$. It follows that $M_1(r+r', \a)\cong M_1(r+r', \eta)$,
which forces $\a=\eta$ by \cite[Theorem 3.10(4)]{Ch4}. This completes the proof.
\end{proof}

In the following, unless otherwise stated, all isomorphisms
are $D_4$-module isomorphisms.

\begin{proposition}\label{2.5}
Let $r, r'\in{\mathbb Z}_2$. Then $V(2, r)\ot V(2, r')\cong P(r+r'+1)$.
\end{proposition}

\begin{proof}
We first show that $V(2, 0)\ot V(2, 0)\cong P(1)$. Let $\{v_1, v_2\}$ be the standard
basis of $V(2, 0)$ as stated in Section 1. Then $\{v_i\ot v_j|1\<i, j\<2\}$ is a basis of
$V(2, 0)\ot V(2, 0)$. Let $u_1=v_2\ot v_1-v_1\ot v_2$, $u_2=2v_2\ot v_2$, $u_3=-4v_1\ot v_1$
and $u_4=-4(v_1\ot v_2 +v_2\ot v_1)$. Then $\{u_1, u_2, u_3, u_4\}$ is also a basis
of $V(2, 0)\ot V(2, 0)$. Now by a straightforward verification, we have
\begin{equation*}
\begin{array}{llll}
a\cdot u_1=u_2,& d\cdot u_1=u_3,& b\cdot u_1=-u_1, & c\cdot u_1=-u_1,\\
a\cdot u_2=0,& d\cdot u_2=-u_4, & b\cdot u_2=u_2, & c\cdot u_2=u_2,\\
a\cdot u_3=u_4, & d\cdot u_3=0, & b\cdot u_3=u_3, &c\cdot u_3=u_3,\\
a\cdot u_4=0, & d\cdot u_4=0,   & b\cdot u_4=-u_4, & c\cdot u_4=-u_4.
\end{array}
\end{equation*}
This shows that $V(2, 0)\ot V(2, 0)\cong P(1)$. Then by Proposition \ref{2.1}, we have
$$\begin{array}{rcl}
V(2, r)\ot V(2, r')&\cong & V(r)\ot V(2, 0)\ot V(r')\ot V(2, 0)\\
&\cong &  V(r)\ot V(r')\ot V(2, 0)\ot V(2, 0)\\
&\cong &  V(r+r')\ot P(1)\\
&\cong & P(r+r'+1).\\
\end{array}$$
\end{proof}

\begin{lemma}\label{2.6}
Let $M$ be an indecomposable $D_4$-module with ${\rm rl}(M)=2$. If $M$ is of $(s, t)$-type and
${\rm soc}(M)\cong tV(r)$ for some $r\in{\mathbb Z}_2$, then
$$V(2, r')\ot M \cong tV(2, r+r')\oplus sV(2, r+r'+1)$$ for any $r'\in{\mathbb Z}_2$.
\end{lemma}

\begin{proof}
Let $M$ be an indecomposable $D_4$-module with ${\rm rl}(M)=2$, and assume that $M$ is of $(s, t)$-type
with ${\rm soc}(M)\cong tV(r)$ for some $r\in{\mathbb Z}_2$. Then
$M/{\rm soc}(M)\cong sV(r+1)$ by the structure of indecomposable $D_4$-modules described in Section 1.
Hence there is an exact sequence of $D_4$-modules
$$0\ra tV(r)\hookrightarrow M\ra sV(r+1)\ra 0.$$
Let $r'\in{\mathbb Z}_2$. Applying $V(2, r')\ot$ to the above sequence, one gets another
exact sequence of $D_4$-modules
\begin{equation}\label{simple2}
0\ra  V(2, r')\ot(tV(r))\hookrightarrow V(2, r')\ot M\ra V(2, r')\ot(sV(r+1))\ra 0.
\end{equation}
From Proposition \ref{2.1}, one knows that $V(2, r')\ot(tV(r))\cong tV(2, r+r')$ and
$V(2, r')\ot(sV(r+1))\cong sV(2, r+r'+1)$. Since $V(2, r+r'+1)$ is projective,
the sequence (\ref{simple2}) is split. It follows that
$V(2, r')\ot M\cong tV(2, r+r')\oplus sV(2, r+r'+1)$.
\end{proof}

\begin{corollary}\label{2.7}
Let $r, r'\in{\mathbb Z}_2$, $s\>1$ and $\eta\in{\mathbb P}^1(k)$. Then we have\\
$(1)$ If $s$ is odd, then
$$V(2, r')\ot\O^sV(r)\cong V(2, r')\ot\O^{-s}V(r)\cong sV(2, r+r')\oplus(s+1)V(2, r+r'+1).$$
$(2)$ If $s$ is even, then
$$V(2, r')\ot\O^sV(r)\cong V(2, r')\ot\O^{-s}V(r)\cong sV(2, r+r'+1)\oplus(s+1)V(2, r+r').$$
$(3)$ $V(2, r')\ot M_s(r, \eta)\cong sV(2, 0)\oplus sV(2, 1)$.
\end{corollary}

\begin{proof}
If $s$ is odd, then soc$(\O^sV(r))\cong \O^{-s}V(r)/{\rm soc}(\O^{-s}V(r))\cong sV(r)$
and $\O^sV(r)/{\rm soc}(\O^sV(r))\cong{\rm soc}(\O^{-s}V(r))\cong (s+1)V(r+1)$.
Hence Part (1) follows from Lemma \ref{2.6}.
If $s$ is even, then soc$(\O^sV(r))\cong \O^{-s}V(r)/{\rm soc}(\O^{-s}V(r))\cong sV(r+1)$
and $\O^sV(r)/{\rm soc}(\O^sV(r))\cong{\rm soc}(\O^{-s}V(r))\cong (s+1)V(r)$.
Hence Part (2) follows from Lemma \ref{2.6}.
Since ${\rm soc}(M_s(r, \eta))\cong sV(r)$ and $M_s(r, \eta)/{\rm soc}(M_s(r, \eta))\cong sV(r+1)$,
it follows from Lemma \ref{2.6} that
$V(2, r')\ot M_s(r, \eta)\cong sV(2, r+r')\oplus sV(2, r+r'+1)\cong sV(2, 0)\oplus sV(2, 1)$.
This shows Part (3).
\end{proof}

\begin{corollary}\label{2.8}
Let $r, r'\in{\mathbb Z}_2$. Then $V(2, r')\ot P(r)\cong 2V(2, 0)\oplus 2V(2, 1)$.
\end{corollary}

\begin{proof}
Applying $V(2, r')\ot$ to the exact sequence
$0\ra\O V(r)\hookrightarrow P(r)\ra V(r)\ra 0$, one gets the following exact sequence of $D_4$-modules
$$0\ra V(2, r')\ot\O V(r)\hookrightarrow V(2, r')\ot P(r)\ra V(2, r')\ot V(r)\ra 0.$$
By Proposition \ref{2.1} and Corollary \ref{2.7}, we have $V(2, r')\ot\O V(r)\cong V(2, r+r')\oplus 2V(2, r+r'+1)$
and $V(2, r')\ot V(r)\cong V(2, r+r')$. Since $V(2, r+r')$ is projective, the above sequence is split.
Hence $V(2, r')\ot P(r)\cong 2V(2, r+r')\oplus 2V(2, r+r'+1)\cong 2V(2, 0)\oplus 2V(2, 1)$.
\end{proof}

\begin{lemma}\label{2.9}
Let $M$ be an indecomposable $D_4$-module with ${\rm rl}(M)=2$. If $M$ is of $(s, t)$-type and
${\rm soc}(M)\cong tV(r)$ for some $r\in{\mathbb Z}_2$, then
$$P(r')\ot M \cong tP(r+r')\oplus sP(r+r'+1)$$ for any $r'\in{\mathbb Z}_2$.
\end{lemma}

\begin{proof}
It is similar to Lemma \ref{2.6}.
\end{proof}

\begin{corollary}\label{2.10}
Let $r, r'\in{\mathbb Z}_2$, $s\>1$ and $\eta\in{\mathbb P}^1(k)$. Then we have\\
$(1)$ If $s$ is odd, then
$$P(r')\ot\O^sV(r)\cong P(r')\ot\O^{-s}V(r)\cong sP(r+r')\oplus(s+1)P(r+r'+1).$$
$(2)$ If $s$ is even, then
$$P(r')\ot\O^sV(r)\cong P(r')\ot\O^{-s}V(r)\cong sP(r+r'+1)\oplus(s+1)P(r+r').$$
$(3)$ $P(r')\ot M_s(r, \eta)\cong sP(0)\oplus sP(1)$.
\end{corollary}

\begin{proof}
It is similar to Corollary \ref{2.7} by using Lemma \ref{2.9}.
\end{proof}

\begin{corollary}\label{2.11}
Let $r, r'\in{\mathbb Z}_2$. Then $P(r')\ot P(r)\cong 2P(0)\oplus 2P(1)$.
\end{corollary}

\begin{proof}
It is similar to Corollary \ref{2.8} by using Proposition \ref{2.1} and Corollary \ref{2.10}.
\end{proof}

For a $D_4$-module $M$, let $M_{(r)}=\{m\in M|b\cdot m=c\cdot m=(-1)^rm\}$, $r\in{\mathbb Z}_2$.
If $M$ is a $D_4$-module without composition factors of dimension 2, then obviously $M=M_{(0)}\oplus M_{(1)}$
as vector spaces. If $M$ and $N$ are $D_4$-modules and
$f: M\ra N$ is a $D_4$-module map, then $f(M_{(r)})\subseteq N_{(r)}$ for any $r\in{\mathbb Z}_2$.

In what follows, we regard $\O^0V(r)=V(r)$ for any $r\in{\mathbb Z}_2$.

\begin{lemma}\label{2.12}
Let $M$ be an indecomposable $D_4$-module with ${\rm rl}(M)=2$, and assume $M/(JM)\cong sV(r)$
for some $s\>1$ and $r\in{\mathbb Z}_2$. If $f: sP(r)\oplus tP(r+1)\ra M$ is a $D_4$-module epimorphism for some
$t\>1$, then ${\rm Ker}(f)\cong \O M\oplus tP(r+1)$.\\
\end{lemma}

\begin{proof}
Assume $f: sP(r)\oplus tP(r+1)\ra M$ is a $D_4$-module epimorphism, $t\>1$.
Let $N=sP(r)$ and $L=tP(r+1)$, and regard $N$ and $L$ as submodules of $N\oplus L$.
Then $f$ induces a $D_4$-module epimorphism
$\ol{f}: (N/(JN))\oplus (L/(JL))\ra M/(JM)$.
Since $N/(JN)\cong M/(JM)\cong sV(r)$ and $L/(JL)\cong tV(r+1)$,
$\ol{f}(L/(JL))=0$, and hence $\ol{f}(N/(JN))=M/(JM)$.
It follows that $f(L)\subseteq JM={\rm soc}(M)$ and
$f|_N: N\ra M$ is surjective. Consequently,
$f|_N: N\ra M$ is a projective cover of $M$,
and so ${\rm Ker}(f|_N)\cong \O M$.

Now let us consider the $D_4$-module map $f|_L: L\ra M$.
Since $L$ is a projective module and $f|_N: N\ra M$ is an epimorphism,
there is a $D_4$-module map $\phi: L\ra N$ such that $(f|_N)\phi=f|_L$.
Define a map $g: L\ra N\oplus L$ by $g(v)=v-\phi(v)$ for all $v\in L$.
It is easy to check that $g$ is a $D_4$-module monomorphism.
Let $L'={\rm Im}(g)$. Then $L'$ is a
submodule of $N\oplus L$ and $L'\cong L\cong tP(r+1)$, and hence ${\rm soc}(L')\cong tV(r+1)$.
However, ${\rm soc}(N)\cong sV(r)$. It follows that $L'\cap N=0$. Thus, we have $N\oplus L=N\oplus L'$
by comparing their lengths. Obviously, $L'\subseteq{\rm Ker}(f)$.
It follows that ${\rm Ker}(f)=({\rm Ker}(f)\cap N)\oplus L'={\rm Ker}(f|_N)\oplus L'
\cong \O M\oplus tP(r+1)$.
\end{proof}

\begin{lemma}\label{2.13}
Let $r, r'\in{\mathbb Z}_2$ and $s\>1$.\\
$(1)$ If $s$ is odd, then
$\O^sV(r)\ot \O V(r')\cong\O^{s+1}V(r+r')\oplus sP(r+r')$.\\
$(2)$ If $s$ is even, then
$\O^sV(r)\ot \O V(r')\cong\O^{s+1}V(r+r')\oplus sP(r+r'+1)$.
\end{lemma}

\begin{proof}
(1) Assume $s$ is odd. Applying $\O^sV(r)\ot$ to the exact sequence
$0\ra \O V(r')\hookrightarrow P(r')\ra V(r')\ra 0$, one gets an exact sequence
of $D_4$-modules
$$0\ra\O^sV(r)\ot\O V(r')\hookrightarrow\O^sV(r)\ot P(r')\ra\O^sV(r)\ot V(r')\ra 0.$$
By Corollaries \ref{2.3} and \ref{2.10}(1), we have $\O^sV(r)\ot V(r')\cong\O^sV(r+r')$ and
$\O^sV(r)\ot P(r')\cong sP(r+r')\oplus(s+1)P(r+r'+1)$. Hence we have an exact sequence
of $D_4$-modules
$$0\ra\O^sV(r)\ot\O V(r')\ra sP(r+r')\oplus(s+1)P(r+r'+1)\ra\O^sV(r+r')\ra 0.$$
Since ${\rm rl}(\O^sV(r+r'))=2$ and $\O^sV(r+r')/(J\O^sV(r+r'))\cong(s+1)V(r+r'+1)$,
it follows from Lemma \ref{2.12} that $\O^sV(r)\ot \O V(r')\cong\O^{s+1}V(r+r')\oplus sP(r+r')$.

(2) It is similar to (1) by using Corollaries \ref{2.3}, \ref{2.10}(2) and Lemma \ref{2.12}.
\end{proof}

\begin{proposition}\label{2.14}
Let $s, t\>1$ and $r, r'\in{\mathbb Z}_2$.\\
$(1)$ It $s+t$ is even, then
$\O^sV(r)\ot \O^tV(r')\cong\O^{s+t}V(r+r')\oplus stP(r+r')$.\\
$(2)$ If $s+t$ is odd, then
$\O^sV(r)\ot \O^t V(r')\cong\O^{s+t}V(r+r')\oplus stP(r+r'+1)$.
\end{proposition}

\begin{proof}
We prove the proposition by induction on $t$.
If $t=1$, then the proposition follows from Lemma \ref{2.13}. Now let $t>1$.

Assume $s+t$ is even. Then $s+t-1$ is odd. By the induction hypothesis, we have
$\O^sV(r)\ot \O^{t-1} V(r')\cong\O^{s+t-1}V(r+r')\oplus s(t-1)P(r+r'+1)$. Then by Lemma \ref{2.13}
and Corollary \ref{2.10}, we have
$$\begin{array}{rl}
&\O^sV(r)\ot \O^{t-1} V(r')\ot\O V(0)\\
\cong & \O^{s+t-1}V(r+r')\ot\O V(0)\oplus s(t-1)P(r+r'+1)\ot\O V(0)\\
\cong & \O^{s+t}V(r+r')\oplus(st+(s+1)(t-1))P(r+r')\oplus s(t-1)P(r+r'+1).\\
\end{array}$$
On the other hand, if $t-1$ is odd, then $s$ is even. In this case,
by Lemma \ref{2.13} and Corollary \ref{2.10}, we have
$$\begin{array}{rl}
&\O^sV(r)\ot \O^{t-1} V(r')\ot\O V(0)\\
\cong & \O^sV(r)\ot\O^t V(r')\oplus\O^sV(r)\ot(t-1)P(r')\\
\cong & \O^{s}V(r)\ot\O^t V(r')\oplus(s+1)(t-1)P(r+r')\oplus s(t-1)P(r+r'+1).\\
\end{array}$$
If $t-1$ is even, then $s$ is odd. In this case,
by Lemma \ref{2.13} and Corollary \ref{2.10}, we have
$$\begin{array}{rl}
&\O^sV(r)\ot \O^{t-1} V(r')\ot\O V(0)\\
\cong & \O^sV(r)\ot\O^t V(r')\oplus\O^sV(r)\ot(t-1)P(r'+1)\\
\cong & \O^{s}V(r)\ot\O^t V(r')\oplus(s+1)(t-1)P(r+r')\oplus s(t-1)P(r+r'+1).\\
\end{array}$$
Thus, we have proved the following isomorphism
$$\begin{array}{rl}
& \O^{s}V(r)\ot\O^t V(r')\oplus(s+1)(t-1)P(r+r')\oplus s(t-1)P(r+r'+1)\\
\cong & \O^{s+t}V(r+r')\oplus(st+(s+1)(t-1))P(r+r')\oplus s(t-1)P(r+r'+1).\\
\end{array}$$
It follows from Krull-Schmidt-Remak Theorem that
$$\O^{s}V(r)\ot\O^t V(r')\cong \O^{s+t}V(r+r')\oplus stP(r+r').$$

Assume $s+t$ is odd.
%Then $s+t-1$ is even. By the induction hypothesis, we have
%$\O^sV(r)\ot \O^{t-1} V(r')\cong\O^{s+t-1}V(r+r')\oplus s(t-1)P(r+r')$. Then by \leref{2.13}
%and \coref{2.10}, we have
%$$\begin{array}{rl}
%&\O^sV(r)\ot \O^{t-1} V(r')\ot\O V(0)\\
%\cong & \O^{s+t-1}V(r+r')\ot\O V(0)\oplus s(t-1)P(r+r')\ot\O V(0)\\
%\cong & \O^{s+t}V(r+r')\oplus(st+(s+1)(t-1))P(r+r'+1)\oplus s(t-1)P(r+r').\\
%\end{array}$$
%On the other hand, if $t-1$ is odd, then so is $s$. In this case,
%by \leref{2.13} and \coref{2.10}, we have
%$$\begin{array}{rl}
%&\O^sV(r)\ot \O^{t-1} V(r')\ot\O V(0)\\
%\cong & \O^sV(r)\ot\O^t V(r')\oplus\O^sV(r)\ot((t-1)P(r'))\\
%\cong & \O^{s}V(r)\ot\O^t V(r')\oplus(s+1)(t-1)P(r+r'+1)\oplus s(t-1)P(r+r').\\
%\end{array}$$
%If $t-1$ is even, then so is $s$. In this case,
%by \leref{2.13} and \coref{2.10}, we have
%$$\begin{array}{rl}
%&\O^sV(r)\ot \O^{t-1} V(r')\ot\O V(0)\\
%\cong & \O^sV(r)\ot\O^t V(r')\oplus\O^sV(r)\ot((t-1)P(r'+1))\\
%\cong & \O^{s}V(r)\ot\O^t V(r')\oplus(s+1)(t-1)P(r+r'+1)\oplus s(t-1)P(r+r').\\
%\end{array}$$
%Thus, we have proved the following isomorphism
Then similarly, one can show that
$$\begin{array}{rl}
& \O^{s}V(r)\ot\O^t V(r')\oplus(s+1)(t-1)P(r+r'+1)\oplus s(t-1)P(r+r')\\
\cong & \O^{s+t}V(r+r')\oplus(st+(s+1)(t-1))P(r+r'+1)\oplus s(t-1)P(r+r').\\
\end{array}$$
From Krull-Schmidt-Remak Theorem, we have
$$\O^{s}V(r)\ot\O^t V(r')\cong \O^{s+t}V(r+r')\oplus stP(r+r'+1).$$
\end{proof}

\begin{lemma}\label{2.15}
Let $r\in{\mathbb Z}_2$. Then $V(r)^*\cong V(r)$ and $V(2,r)^*\cong V(2, r+1)$.
\end{lemma}

\begin{proof}
It follows from a straightforward verification.
\end{proof}

By Lemma \ref{2.15}, one can check the following lemma.

\begin{lemma}\label{2.16}
Let  $r\in{\mathbb Z}_2$. Then $P(r)^*\cong P(r)$, $(\O^sV(r))^*\cong\O^{-s}V(r)$ and
$(\O^{-s}V(r))^*\cong\O^{s}V(r)$
for all $s\>1$.
\end{lemma}

\begin{proof}
From the discussion in Section 1, there is a minimal projective resolution of $V(r)$:
$$
\cdots\ra4P(r+1)\ra3P(r)\ra2P(r+1)\ra P(r)\ra
V(r)\ra 0.$$
Applying the duality $(-)^*$ to the above resolution,
one gets a minimal injective resolution of $V(r)^*$:
$$0\ra V(r)^*\ra P(r)^*\ra 2P(r+1)^*\ra 3P(r)^*\ra 4P(r+1)^*
\ra\cdots.$$
By Lemma \ref{2.15}, $V(r)^*\cong V(r)$. It follows from the discussion in Section 1 that
$P(r)^*\cong P(r)$ and $(\O^sV(r))^*\cong\O^{-s}V(r)$
for all $s\>1$. Then $(\O^{-s}V(r))^*\cong(\O^{s}V(r))^{**}\cong \O^sV(r)$
for all $s\>1$.
\end{proof}

\begin{corollary}\label{2.17}
Let $s, t\>1$ and $r, r'\in{\mathbb Z}_2$.\\
$(1)$ It $s+t$ is even, then
$\O^{-s}V(r)\ot \O^{-t}V(r')\cong\O^{-(s+t)}V(r+r')\oplus stP(r+r')$.\\
$(2)$ If $s+t$ is odd, then
$\O^{-s}V(r)\ot \O^{-t}V(r')\cong\O^{-(s+t)}V(r+r')\oplus stP(r+r'+1)$.
\end{corollary}

\begin{proof}
We have already known that $(M\ot N)^*\cong N^*\ot M^*\cong M^*\ot N^*$ for any
$M, N\in{\rm mod}D_4$. Thus, the corollary follows from Proposition \ref{2.14} and
Lemma \ref{2.16} by using the duality $(-)^*$.
\end{proof}

\begin{lemma}\label{2.18}
Let $r, r'\in{\mathbb Z}_2$ and $s\>1$.\\
$(1)$ If $s$ is odd, then $\O^{-s}V(r)\ot \O V(r')\cong\O^{-s+1}V(r+r')\oplus (s+1)P(r+r'+1)$.\\
$(2)$ If $s$ is even, then $\O^{-s}V(r)\ot \O V(r')\cong\O^{-s+1}V(r+r')\oplus (s+1)P(r+r')$.\\
\end{lemma}

\begin{proof}
Applying $\O^{-s}V(r)\ot$ to the exact sequence
$0\ra\O V(r')\hookrightarrow P(r')\ra V(r')\ra 0$,
one gets the following exact sequence
of $D_4$-modules
$$0\ra\O^{-s}V(r)\ot\O V(r')\hookrightarrow\O^{-s}V(r)\ot P(r')\ra\O^{-s}V(r)\ot V(r')\ra 0.$$
From Corollary \ref{2.3}, we have $\O^{-s}V(r)\ot V(r')\cong \O^{-s}V(r+r')$.
From Corollary \ref{2.10}, we have that
$\O^{-s}V(r)\ot P(r')\cong sP(r+r')\oplus(s+1)P(r+r'+1)$ if $s$ is odd, and that
$\O^{-s}V(r)\ot P(r')\cong sP(r+r'+1)\oplus(s+1)P(r+r')$ if $s$ is even.
Then the lemma follows from the above exact sequence and Lemma \ref{2.12}.
\end{proof}

\begin{corollary}\label{2.19}
Let $r, r'\in{\mathbb Z}_2$ and $s\>1$.\\
$(1)$ If $s$ is odd, then $\O^{s}V(r)\ot \O^{-1} V(r')\cong(s+1)P(r+r'+1)\oplus\O^{s-1}V(r+r')$.\\
$(2)$ If $s$ is even, then $\O^{s}V(r)\ot \O^{-1}V(r')\cong(s+1)P(r+r')\oplus\O^{s-1}V(r+r')$.\\
\end{corollary}

\begin{proof}
It follows from Lemmas \ref{2.16} and \ref{2.18} by using the duality $(-)^*$.
\end{proof}

\begin{proposition}\label{2.20}
Let $r, r'\in{\mathbb Z}_2$ and $s, t\>1$.\\
$(1)$ If $s+t$ is even and $s\>t$, then
$$\O^{s}V(r)\ot \O^{-t} V(r')\cong \O^{s-t}V(r+r')\oplus(s+1)tP(r+r'+1).$$
$(2)$ If $s+t$ is odd and $s\>t$, then
$$\O^{s}V(r)\ot \O^{-t}V(r')\cong \O^{s-t}V(r+r')\oplus(s+1)tP(r+r').$$
$(3)$ If $s+t$ is even and $s<t$, then
$$\O^{s}V(r)\ot \O^{-t} V(r')\cong \O^{s-t}V(r+r')\oplus(t+1)sP(r+r'+1).$$
$(4)$ If $s+t$ is odd and $s<t$, then
$$\O^{s}V(r)\ot \O^{-t}V(r')\cong \O^{s-t}V(r+r')\oplus(t+1)sP(r+r').$$
\end{proposition}

\begin{proof}
We first prove Parts (1) and (2) by induction on $t$. If $t=1$, they follow from Corollary \ref{2.19}.
Now let $s\>t>1$.

Assume $s+t$ is even. Then both $s+t-1$ and $s-t+1$ are odd, and hence
$\O^{s}V(r)\ot \O^{-t+1}V(r')\cong(s+1)(t-1)P(r+r')\oplus\O^{s-t+1}V(r+r')$ by the induction hypothesis.
Thus, by Corollaries \ref{2.10}(1) and \ref{2.19}(1), we have
$$\begin{array}{rl}
&\O^{s}V(r)\ot \O^{-t+1}V(r')\ot\O^{-1}V(0)\\
\cong & (s+1)(t-1)P(r+r')\ot\O^{-1}V(0)\oplus\O^{s-t+1}V(r+r')\ot\O^{-1}V(0)\\
\cong & (s+1)(t-1)P(r+r')\oplus (2st-s+t)P(r+r'+1)\oplus\O^{s-t}V(r+r').\\
\end{array}$$
On the other hand, if $t-1$ is odd, then $s$ is even. In this case, by Corollaries \ref{2.17}(1)
and \ref{2.10}(2), we have
$$\begin{array}{rl}
&\O^{s}V(r)\ot \O^{-t+1}V(r')\ot\O^{-1}V(0)\\
\cong & \O^{s}V(r)\ot\O^{-t}V(r')\oplus\O^{s}V(r)\ot(t-1)P(r')\\
\cong & \O^{s}V(r)\ot\O^{-t}V(r')\oplus s(t-1)P(r+r'+1)\oplus (s+1)(t-1)P(r+r').\\
\end{array}$$
If $t-1$ is even, then $s$ is odd. In this case, by Corollaries \ref{2.17}(2)
and \ref{2.10}(1), we have
$$\begin{array}{rl}
&\O^{s}V(r)\ot \O^{-t+1}V(r')\ot\O^{-1}V(0)\\
\cong & \O^{s}V(r)\ot\O^{-t}V(r')\oplus\O^{s}V(r)\ot(t-1)P(r'+1)\\
\cong & \O^{s}V(r)\ot\O^{-t}V(r')\oplus s(t-1)P(r+r'+1)\oplus (s+1)(t-1)P(r+r').\\
\end{array}$$
Thus, we have proved that
$$\begin{array}{rl}
& \O^{s}V(r)\ot\O^{-t}V(r')\oplus s(t-1)P(r+r'+1)\oplus (s+1)(t-1)P(r+r')\\
\cong & \O^{s-t}V(r+r')\oplus (2st-s+t)P(r+r'+1)\oplus(s+1)(t-1)P(r+r').\\
\end{array}$$
Then by Krull-Schmidt-Remak Theorem, we have
$$\O^{s}V(r)\ot\O^{-t}V(r')
\cong\O^{s-t}V(r+r')\oplus (s+1)tP(r+r'+1).$$

Assume $s+t$ is odd.
%Then both $s+t-1$ and $s-t+1$ are even, and hence
%$\O^{s}V(r)\ot \O^{-t+1}V(r')\cong(s+1)(t-1)P(r+r'+1)\oplus\O^{s-t+1}V(r+r')$ by the induction hypothesis.
%Thus, by \coref{2.10}(1) and \coref{2.19}(2), we have
%$$\begin{array}{rl}
%&\O^{s}V(r)\ot \O^{-t+1}V(r')\ot\O^{-1}V(0)\\
%\cong & (s+1)(t-1)P(r+r'+1)\ot\O^{-1}V(0)\oplus\O^{s-t+1}V(r+r')\ot\O^{-1}V(0)\\
%\cong & (s+1)(t-1)P(r+r'+1)\oplus (2st-s+t)P(r+r')\oplus\O^{s-t}V(r+r').\\
%\end{array}$$
%On the other hand, if $t-1$ is odd, then $s$ is odd. In this case, by \coref{2.17}(1)
%and \coref{2.10}(1), we have
%$$\begin{array}{rl}
%&\O^{s}V(r)\ot \O^{-t+1}V(r')\ot\O^{-1}V(0)\\
%\cong & \O^{s}V(r)\ot\O^{-t}V(r')\oplus\O^{s}V(r)\ot(t-1)P(r')\\
%\cong & \O^{s}V(r)\ot\O^{-t}V(r')\oplus s(t-1)P(r+r')\oplus (s+1)(t-1)P(r+r'+1).\\
%\end{array}$$
%If $t-1$ is even, then $s$ is even. In this case, by \coref{2.17}(2)
%and \coref{2.10}(2), we have
%$$\begin{array}{rl}
%&\O^{s}V(r)\ot \O^{-t+1}V(r')\ot\O^{-1}V(0)\\
%\cong & \O^{s}V(r)\ot\O^{-t}V(r')\oplus\O^{s}V(r)\ot(t-1)P(r'+1)\\
%\cong & \O^{s}V(r)\ot\O^{-t}V(r')\oplus s(t-1)P(r+r')\oplus (s+1)(t-1)P(r+r'+1).\\
%\end{array}$$
Following the argument above, one can show that
$$\begin{array}{rl}
& \O^{s}V(r)\ot\O^{-t}V(r')\oplus s(t-1)P(r+r')\oplus (s+1)(t-1)P(r+r'+1)\\
\cong & \O^{s-t}V(r+r')\oplus (2st-s+t)P(r+r')\oplus(s+1)(t-1)P(r+r'+1).\\
\end{array}$$
Then by Krull-Schmidt-Remak Theorem, we have
$$\O^{s}V(r)\ot\O^{-t}V(r')
\cong\O^{s-t}V(r+r')\oplus (s+1)tP(r+r').$$

Thus, we have proved Parts (1) and (2).

Now assume that $s+t$ is even and $s<t$. Then by Part (1), we have
$\O^{-s} V(r)\ot \O^{t}V(r')\cong(t+1)sP(r+r'+1)\oplus\O^{t-s}V(r+r')$.
Applying the duality $(-)^*$ to the isomorphism, it follows from Lemma \ref{2.16} that
$$\O^{s} V(r)\ot\O^{-t}V(r')\cong\O^{s-t}V(r+r')\oplus(t+1)sP(r+r'+1).$$
This shows Part (3). Similarly, Part (4) follows from Part (2) and Lemma \ref{2.16}
by using the duality $(-)^*$.
\end{proof}

\begin{proposition}\label{2.21}
Let $s, t\>1$ and $\eta\in{\mathbb P}^1(k)$.\\
$(1)$ If $s$ is odd then $M_t(0, \eta)\ot \O^sV(0)\cong stP(0)\oplus M_t(1, \eta)$.\\
$(2)$ If $s$ is even then $M_t(0, \eta)\ot \O^sV(0)\cong stP(1)\oplus M_t(0, \eta)$.
\end{proposition}

\begin{proof}
We prove the proposition by induction on $s$.

Applying $M_t(0, \eta)\ot$ to the exact sequence
$0\ra\O V(0)\hookrightarrow P(0)\ra V(0)\ra 0$,
one gets the following exact sequence of $D_4$-modules
$$0\ra M_t(0, \eta)\ot\O V(0)\hookrightarrow M_t(0, \eta)\ot P(0)\ra M_t(0, \eta)\ot V(0)\ra 0.$$
By Proposition \ref{2.4} and Corollary \ref{2.10}(3), we have $M_t(0, \eta)\ot V(0)\cong M_t(0, \eta)$
and $M_t(0, \eta)\ot P(0)\cong tP(0)\oplus tP(1)$. Hence we have an exact sequence
$$0\ra M_t(0, \eta)\ot\O V(0)\ra tP(0)\oplus tP(1)\ra M_t(0, \eta)\ra 0.$$
It follows from Lemma \ref{2.12} that $M_t(0, \eta)\ot\O V(0)\cong tP(0)\oplus\O M_t(0, \eta)
\cong tP(0)\oplus M_t(1, \eta)$.

Let $s>1$ be even. Then we have an exact sequence
$$0\ra\O^s V(0)\hookrightarrow sP(1)\ra \O^{s-1}V(0)\ra 0.$$
Applying $M_t(0, \eta)\ot$ to the above exact sequence, one get the following exact sequence
$$0\ra M_t(0, \eta)\ot\O^s V(0)\ra M_t(0, \eta)\ot(sP(1))\ra M_t(0, \eta)\ot\O^{s-1}V(0)\ra 0.$$
By Lemma \ref{2.9}, we have $M_t(0, \eta)\ot(sP(1))\cong stP(0)\oplus stP(1)$. By the induction
hypothesis, we have $M_t(0, \eta)\ot\O^{s-1}V(0)\cong (s-1)tP(0)\oplus M_t(1,\eta)$.
Hence we have an exact sequence
$$0\ra M_t(0, \eta)\ot\O^s V(0)\ra stP(0)\oplus stP(1)\ra (s-1)tP(0)\oplus M_t(1,\eta)\ra 0.$$
Since $(s-1)tP(0)$ is projective, from the above exact sequence, one can deduce the following exact sequence
$$0\ra M_t(0, \eta)\ot\O^s V(0)\ra tP(0)\oplus stP(1)\ra M_t(1,\eta)\ra 0.$$
It follows from Lemma \ref{2.12} that $M_t(0, \eta)\ot\O^s V(0)\cong stP(1)\oplus \O M_t(1,\eta)
\cong stP(1)\oplus M_t(0,\eta)$.

Let $s>1$ be odd. Then we have an exact sequence
$$0\ra\O^s V(0)\hookrightarrow sP(0)\ra \O^{s-1}V(0)\ra 0.$$
Then an argument similar to the above one shows that
%Applying $M_1(0, \eta)\ot$ to the above exact sequence, one get the following exact sequence
%$$0\ra M_1(0, \eta)\ot\O^s V(0)\ra M_1(0, \eta)\ot(sP(0))\ra M_1(0, \eta)\ot\O^{s-1}V(0)\ra 0.$$
%Again by \leref{2.9}, we have $M_1(0, \eta)\ot(sP(1))\cong sP(0)\oplus sP(1)$. By the induction
%hypothesis, we have $M_1(0, \eta)\ot\O^{s-1}V(0)\cong (s-1)P(1)\oplus M_1(0,\eta)$.
%Hence we have an exact sequence
%$$0\ra M_1(0, \eta)\ot\O^s V(0)\ra sP(0)\oplus sP(1)\ra (s-1)P(1)\oplus M_1(0,\eta)\ra 0.$$
%Since $(s-1)P(1)$ is projective, one can deduce an exact sequence, from the above one, as follows
%$$0\ra M_1(0, \eta)\ot\O^s V(0)\ra sP(0)\oplus P(1)\ra M_1(0,\eta)\ra 0.$$
%It follows from \leref{2.12} that
$M_t(0, \eta)\ot\O^s V(0)\cong stP(0)\oplus \O M_t(0,\eta)
\cong stP(0)\oplus M_t(1,\eta)$.
\end{proof}

\begin{corollary}\label{2.22}
Let $s, t\>1$, $r, r'\in{\mathbb Z}_2$ and $\eta\in{\mathbb P}^1(k)$.\\
$(1)$ If $s$ is odd then $M_t(r, \eta)\ot \O^sV(r')\cong stP(r+r')\oplus M_t(r+r'+1, \eta)$.\\
$(2)$ If $s$ is even then $M_t(r, \eta)\ot \O^sV(r')\cong stP(r+r'+1)\oplus M_t(r+r', \eta)$.
\end{corollary}

\begin{proof}
If $s$ is odd, then by Propositions \ref{2.1}, \ref{2.4}, \ref{2.21}(1) and Corollary \ref{2.3}, we have
$$\begin{array}{rcl}
M_t(r, \eta)\ot \O^sV(r')&\cong & V(r)\ot M_t(0, \eta)\ot V(r')\ot \O^sV(0)\\
&\cong & V(r+r')\ot M_t(0, \eta)\ot \O^sV(0)\\
&\cong & V(r+r')\ot(stP(0)\oplus M_t(1, \eta))\\
&\cong & stP(r+r')\oplus M_t(r+r'+1, \eta).\\
\end{array}$$
This shows Part (1). Part (2) can be shown similarly.
\end{proof}

\begin{lemma}\label{2.23}
Let $s\>1$, $r\in{\mathbb Z}_2$ and $\eta\in{\mathbb P}^1(k)$. Then $M_s(r,\eta)^*\cong M_s(r+1, \eta)$.
\end{lemma}

\begin{proof}
By Proposition \ref{2.4} and Lemma \ref{2.15}, we only need to show that $M_s(0,\eta)^*\cong M_s(1, \eta)$.

For $\eta\in k$, let $\{v_1, v_2\}$ be the standard basis of $M_1(0,\eta)$ as stated in Section 1,
and let $\{f_1, f_2\}$ be the dual basis in $M_1(0,\eta)^*$. Then one can easily check that
\begin{equation*}
\begin{array}{llll}
a\cdot f_2=f_1, & d\cdot f_2=-\eta f_1, & b\cdot f_2=f_2, & c\cdot f_2=f_2,\\
a\cdot f_1=0, & d\cdot f_1=0, & b\cdot f_1=-f_1, & c\cdot f_1=-f_1.\\
\end{array}
\end{equation*}
Hence $M_1(0,\eta)^*\cong M_1(1, \eta)$. For $\eta=\infty$,
one can similarly show that $M_1(0,\eta)^*\cong M_1(1, \eta)$.

Now assume $s>1$ and $\eta\in{\mathbb P}^1(k)$. Then there is a $D_4$-module epimorphism
$M_s(0, \eta)\ra M_1(0, \eta)$. Applying the duality $(-)^*$, one gets a $D_4$-module
monomorphism $M_1(0, \eta)^*\ra M_s(0, \eta)^*$. Hence $M_s(0, \eta)^*$ contains a submodule
isomorphic to $M_1(0, \eta)^*\cong M_1(1, \eta)$. It follows from
\cite[Theorem 3.10(2) and Proposition 3.11]{Ch4} that $M_s(0, \eta)^*\cong M_s(1, \eta)$
since $M_s(0, \eta)^*$ is of $(s,s)$-type.
\end{proof}

\begin{corollary}\label{2.24}
Let $s, t\>1$, $r, r'\in{\mathbb Z}_2$ and $\eta\in{\mathbb P}^1(k)$.\\
$(1)$ If $s$ is odd then $M_t(r, \eta)\ot \O^{-s}V(r')\cong stP(r+r'+1)\oplus M_t(r+r'+1, \eta)$.\\
$(2)$ If $s$ is even then $M_t(r, \eta)\ot \O^{-s}V(r')\cong stP(r+r')\oplus M_t(r+r', \eta)$.
\end{corollary}

\begin{proof}
It follows from Lemmas \ref{2.16} and \ref{2.23} by applying the duality $(-)^*$ to the isomorphisms
in Corollary \ref{2.22}.
\end{proof}

\begin{proposition}\label{2.25}
Let $r, r'\in{\mathbb Z}_2$ and $\a, \eta\in{\mathbb P}^1(k)$.
If $\a\neq\eta$ then $M_s(r,\a)\ot M_t(r', \eta)\cong stP(r+r')$ for all $s, t\>1$.
\end{proposition}

\begin{proof}
By Propositions \ref{2.1} and \ref{2.4}, it is enough to show that
$M_s(0,\a)\ot M_t(0, \eta)\cong stP(0)$ for all $s, t\>1$, and $\a\neq\eta$ in
${\mathbb P}^1(k)$. We prove the statement by induction on $s+t$.

Let $\a, \eta\in k$ with $\a\neq \eta$. Let $\{v_1, v_2\}$ and $\{u_1, u_2\}$ be the standard bases
of $M_1(0, \a)$ and $M_1(0, \eta)$ as stated in Section 1, respectively.
Putting $w_1=v_1\ot u_1$, $w_2=v_1\ot u_2-v_2\ot u_1$, $w_3=\a v_2\ot u_1-\eta v_1\ot u_2$
and $w_4=(\a-\eta)v_2\ot u_2$ in $M_1(0, \a)\ot M_1(0, \eta)$.
Since $\a\neq \eta$, $\{w_1, w_2, w_3, w_4\}$ forms a basis of $M_1(0, \a)\ot M_1(0, \eta)$.
Now one can easily check that
\begin{equation*}
\begin{array}{llll}
a\cdot w_1=w_2,& d\cdot w_1=w_3,& b\cdot w_1=w_1, & c\cdot w_1=w_1,\\
a\cdot w_2=0,& d\cdot w_2=-w_4, & b\cdot w_2=-w_2, & c\cdot w_2=-w_2,\\
a\cdot w_3=w_4, & d\cdot w_3=0, & b\cdot w_3=-w_3, & c\cdot w_3=-w_3,\\
a\cdot w_4=0, & d\cdot w_4=0,   & b\cdot w_4=w_4, & c\cdot w_4=w_4.
\end{array}
\end{equation*}
It follows that $M_1(0,\a)\ot M_1(0, \eta)\cong P(0)$. Similarly, one can show
that $M_1(0,\infty)\ot M_1(0, \eta)\cong P(0)$ for all $\eta\in k$.

Now let $\a\neq\eta$ in ${\mathbb P}^1(k)$ and assume $s+t>2$.
We may assume $t>1$ since $M\ot N\cong N\ot M$ for any modules $M$ and $N$.
Then we have an exact sequence
$0\ra M_{t-1}(0, \eta)\ra M_t(0, \eta)\ra M_1(0, \eta)\ra 0$.
Applying $M_s(0, \a)\ot$, one gets the following exact sequence
$$0\ra M_s(0, \a)\ot M_{t-1}(0, \eta)\ra M_s(0, \a)\ot M_t(0, \eta)\ra M_s(0, \a)\ot M_1(0, \eta)\ra 0.$$
By the induction hypothesis, we have $M_s(0, \a)\ot M_{t-1}(0, \eta)\cong s(t-1)P(0)$ and
$M_s(0, \a)\ot M_{1}(0, \eta)\cong sP(0)$. Hence we have an exact sequence
$$0\ra s(t-1)P(0)\ra M_s(0, \a)\ot M_t(0, \eta)\ra sP(0)\ra 0,$$
which is split since $P(0)$ is projective (injective). It follows that
$M_s(0,\a)\ot M_t(0, \eta)\cong stP(0)$.
\end{proof}

\begin{lemma}\label{2.26}
Let $s>i\>1$. If $M$ is an indecomposable module of $(s+1, s)$-type, then $M$ contains no submodules
of $(i+1, i)$-type, and consequently, $M$ contains no proper submodule $N$ with $l(N/{\rm soc}(N))>l(N)$.
\end{lemma}

\begin{proof}
It is similar to \cite[Lemma 4.3]{Ch3}.
\end{proof}

\begin{lemma}\label{2.26-2}
Let $M$ be an indecomposable module of $(s, s)$-type with $s\>2$. Then $M$ contains no submodules
of $(i+1, i)$-type. Consequently, $M$ contains no submodule $N$ with $l(N/{\rm soc}(N))>l(N)$.
\end{lemma}

\begin{proof}
It follows from \cite[Lemma 4.3]{Ch3} and \cite[Proposition 3.3]{Ch4}. It also can be shown by an
argument similar to the proof of \cite[Lemma 4.3]{Ch3}.
\end{proof}

\begin{lemma}\label{2.27}
Let $s\>1$ and $M$ be an indecomposable module of $(s, s)$-type. Then $M$ can be embedded into an indecomposable
module of $(s+1, s)$-type.
\end{lemma}

\begin{proof}
By Corollary \ref{2.3} and Proposition \ref{2.4},
we may assume $M=M_s(0,\eta)$, where $\eta\in{\mathbb P}^1(k)$.
It is enough to show that there is a monomorphism $\s_s: M_s(0, \eta)\ra \O^sV$, where $V=V(0)$ for
$s$ being odd and $V=V(1)$ for $s$ being even. We prove the statement by induction on $s$.

Obviously, there is a $D_4$-module monomorphism $\s_1: M_1(0, \eta)\rightarrow\O V(0)$,
which is not split monomorphism since $\O V(0)$ is indecomposable.
From \cite[Theorem 3.10(5)]{Ch4}, there is an almost split sequence:
$$0\ra M_1(0, \eta)\xrightarrow{\tau} M_2(0, \eta)\ra M_1(0, \eta)\ra 0.$$
Hence $\s_1$ factors through $\tau$, i.e., there is a module map $\phi: M_2(0, \eta)\ra\O V(0)$
such that $\s_1=\phi\tau$. We claim that $\phi$ is an epimorphism.
In fact, we have $\O V(0)\supseteq{\rm Im}(\phi)\supseteq{\rm Im}(\s_1)\cong M_1(0, \eta)$.
Since $l(M_1(0, \eta))=2$ and $l(\O V(0))=3$, ${\rm Im}(\phi)={\rm Im}(\s_1)$ or
${\rm Im}(\phi)=\O V(0)$. If ${\rm Im}(\phi)={\rm Im}(\s_1)$, then $\phi$ can be regarded as
an epimorphism $\phi: M_2(0, \eta)\ra{\rm Im}(\phi)\cong M_1(0, \eta)$, which forces ${\rm Ker}(\phi)$ is of
$(1, 1)$-type. It follows from \cite[Theorem 3.10(2)]{Ch4} that ${\rm Ker}(\phi)={\rm Im}(\tau)$.
Thus, $\phi\tau=0\neq\s_1$, a contradiction. Hence $\phi$ is an epimorphism from $M_2(0, \eta)$ to $\O V(0)$.
From \cite[Theorem 3.5(2)]{Ch4}, there is an almost split sequence
$$0\ra\O^3V(0)\ra \O^2V(1)\oplus\O^2V(1)\xrightarrow{(f, g)}\O V(0)\ra 0.$$
It follows from \cite[Lemma V.5.1]{ARS} that both $f$ and $g$ are epimorphism since they are irreducible morphisms and
$l(\O^2V(1))>l(\O V(0))$.
Note that $\phi$ is not split epimorphism since $M_2(0, \eta)$ is indecomposable.
Hence $\phi$ factors through $(f,g)$, i.e., there is a $D_4$-module map
$\left(
   \begin{array}{c}
     \s_2 \\
     \s'_2\\
   \end{array}
 \right): M_2(0, \eta)\ra \O^2V(1)\oplus\O^2V(1)$ such that $f\s_2+g\s'_2=\phi$.
We first show that $l({\rm Im}(f\s_2))\neq 1$. In fact, if $l({\rm Im}(f\s_2))=1$ then
${\rm Im}(f\s_2)\subseteq{\rm soc}(\O V(0))\cong V(0)$, which implies that
${\rm Im}(f\s_2)\cong V(0)$. Hence $f\s_2$ induces an epimorphism
$\ol{f\s_2}: M_2(0, \eta)/(JM_2(0, \eta))\ra V(0)$. This is impossible since
$M_2(0, \eta)/(JM(0, \eta))\cong 2V(1)$. Similarly, $l({\rm Im}(g\s_2'))\neq 1$.
Without losing generality, we may assume $l({\rm Im}(f\s_2))\>l({\rm Im}(g\s_2'))$.
Since $\phi$ is an epimorphism, we have
$\O V(0)={\rm Im}(\phi)={\rm Im}(f\s_2+g\s'_2)\subseteq {\rm Im}(f\s_2)+{\rm Im}(g\s_2')$.
If ${\rm Im}(g\s_2')=0$ then ${\rm Im}(f\s_2)=\O V(0)$, and so
$f\s_2$ is an epimorphism from $M_2(0, \eta)$ to $\O V(0)$.
Now assume that ${\rm Im}(g\s_2')\neq 0$. Then $3=l(\O V(0))\>l({\rm Im}(f\s_2))\>l({\rm Im}(g\s_2'))\>2$.
If $l({\rm Im}(f\s_2))=2$ then $l({\rm Im}(g\s_2'))=2$. In this case,
$[{\rm Im}(f\s_2)]=[{\rm Im}(g\s'_2)]=[V(0)]+[V(1)]$ in $G_0(D_4)$, and consequently
$[{\rm Ker}(f\s_2)]=[{\rm Ker}(g\s'_2)]=[V(0)]+[V(1)]$ in $G_0(D_4)$ since
$[M_2(0, \eta)]=2[V(0)]+2[V(1)]$. This implies that both ${\rm Ker}(f\s_2)$ and ${\rm Ker}(g\s'_2)$
are submodules of $(1, 1)$-type in $M_2(0, \eta)$. It follows from \cite[Theorem 3.10(2)]{Ch4}
that ${\rm Ker}(f\s_2)={\rm Ker}(g\s'_2)$. This implies that
${\rm Ker}(f\s_2)={\rm Ker}(g\s'_2)\subseteq {\rm Ker}(\phi)\cong V(0)$, a contradiction.
Thus, we have proved that $f\s_2$ is an epimorphism from $M_2(0, \eta)$ to $\O V(0)$.
It follows that ${\rm Ker}(f\s_2)\cong V(0)$. Since ${\rm Ker}(\s_2)\subseteq {\rm Ker}(f\s_2)$,
${\rm Ker}(\s_2)\cong V(0)$ or ${\rm Ker}(\s_2)=0$. If ${\rm Ker}(\s_2)\cong V(0)$, then
${\rm Ker}(\s_2)={\rm Ker}(f\s_2)$, and hence ${\rm Im}(\s_2)\cap{\rm Ker}(f)=0$,
which implies that $\O^2 V(1)={\rm Im}(\s_2)\oplus {\rm Ker}(f)$ by comparing the lengths
of the both sides since $f: \O^2V(1)\ra \O V(0)$ is an epimorphism. This is impossible
since $\O^2V(1)$ is indecomposable. Hence ${\rm Ker}(\s_2)=0$, and so
$\s_2: M_2(0, \eta)\ra \O^2V(1)$ is a monomorphism.

Now let $s>2$ and assume that there is a monomorphism $\s_i: M_i(0, \eta)\ra\O^iV(0)$
for any $1\<i<s$ with $i$ being odd, and there is a monomorphism
$\s_i: M_i(0, \eta)\ra\O^iV(1)$ for any $1\<i<s$ with $i$ being even.
By \cite[Theorem 3.10(5)]{Ch4}, there is an almost split sequence:
$$0\ra M_{s-1}(0, \eta)\xrightarrow{\left(
                                      \begin{array}{c}
                                        g_1 \\
                                        f_1 \\
                                      \end{array}
                                    \right)
}M_{s-2}(0, \eta)\oplus M_s(0, \eta)\ra M_{s-1}(0, \eta)\ra 0.$$
Then $g_1$ is an epimorphism and $f_1$ is a monomorphism since they are irreducible morphisms
and $l(M_s(0, \eta))>l(M_{s-1}(0, \eta))>l(M_{s-2}(0, \eta))$.

Assume $s$ is odd. Then $s-1$ is even. By the induction hypothesis, there is a monomorphism
$\s_{s-1}: M_{s-1}(0, \eta)\ra\O^{s-1}V(1)$, which is not a split monomorphism since
$\O^{s-1}V(1)$ is indecomposable. Hence $\s_{s-1}$ factors through
$\left(\begin{array}{c}
       g_1 \\
       f_1 \\
       \end{array}\right)$, i.e., there is a $D_4$-module map
$(\psi, \phi'): M_{s-2}(0, \eta)\oplus M_s(0, \eta)\ra\O^{s-1}V(1)$ such that
$\psi g_1+\phi' f_1=\s_{s-1}$. We claim that $\phi' f_1: M_{s-1}(0, \eta)\ra \O^{s-1}V(1)$ is injective.
In fact, let $N_i$ be the submodule of $(i, i)$-type in $M_{s-1}(0, \eta)$ for all $1\<i\<s-1$.
Then by \cite[Theorem 3.10(2)]{Ch4}, $N_1\subset N_2\subset\cdots\subset N_{s-1}=M_{s-1}(0, \eta)$
and $N_i\cong M_i(0, \eta)$ for all $1\<i\< s-1$. Moreover, ${\rm Ker}(g_1)=N_1$.
Hence $(\phi' f_1)(N_1)=(\psi g_1+\phi' f_1)(N_1)=\s_{s-1}(N_1)$. Since $\s_{s-1}$ is injective,
${\rm Ker}(\phi' f_1)\cap N_1=0$. If ${\rm Ker}(\phi' f_1)\neq 0$, then there is an $i$ with
$1\<i<s-1$ such that ${\rm Ker}(\phi' f_1)\cap N_i=0$ but ${\rm Ker}(\phi' f_1)\cap N_{i+1}\neq 0$.
Thus, the sum $N_i+({\rm Ker}(\phi' f_1)\cap N_{i+1})$ is a direct sum and is a submodule of $N_{i+1}$,
which implies that $1\<l({\rm Ker}(\phi' f_1)\cap N_{i+1})\< l(N_{i+1})-l(N_i)=2$.
If $l({\rm Ker}(\phi' f_1)\cap N_{i+1})=2$, then $N_{i+1}=N_i\oplus ({\rm Ker}(\phi' f_1)\cap N_{i+1})$.
This is impossible since $N_{i+1}\cong M_{i+1}(0, \eta)$ is indecomposable.
Hence $l({\rm Ker}(\phi' f_1)\cap N_{i+1})=1$, and so
${\rm Ker}(\phi' f_1)\cap N_{i+1}\subseteq{\rm soc}(N_{i+1})$, which implies that
${\rm Ker}(\phi' f_1)\cap N_{i+1}\cong V(0)$. Hence
$[(\phi' f_1)(N_{i+1})]=[N_{i+1}]-[{\rm Ker}(\phi' f_1)\cap N_{i+1}]=(i+1)[V(1)]+i[V(0)]$ in $G_0(D_4)$.
It follows that $l((\phi' f_1)(N_{i+1})/{\rm soc}((\phi' f_1)(N_{i+1})))>l({\rm soc}((\phi' f_1)(N_{i+1})))$
since ${\rm soc}(\O^{s-1}V(1))\cong (s-1)V(0)$ and $\O^{s-1}V(1)/{\rm soc}(\O^{s-1}V(1))\cong sV(1)$.
This contradicts Lemma \ref{2.26}. Hence ${\rm Ker}(\phi' f_1)=0$, which shows the claim that
$\phi' f_1: M_{s-1}(0, \eta)\ra \O^{s-1}V(1)$ is injective.
Then $\phi'({\rm Im}(f_1))={\rm Im}(\phi' f_1)\cong M_{s-1}(0, \eta)\cong{\rm Im}(f_1)$
since $f_1$ is injective. Thus, $l(\phi'({\rm Im}(f_1)))=l({\rm Im}(f_1))$, and so
${\rm Im}(f_1)\cap{\rm Ker}(\phi')=0$. It follows that the sum ${\rm Im}(f_1)+{\rm Ker}(\phi')$ is a direct sum.
Note that ${\rm Im}(\phi' f_1)\subseteq{\rm Im}(\phi')$. It follows that
$2(s-1)=l(M_{s-1}(0, \eta))=l({\rm Im}(\phi' f_1))\<l({\rm Im}(\phi'))\<l(\O^{s-1}V(1))=2s-1$,
and so $l({\rm Im}(\phi'))=2(s-1)$ or $l({\rm Im}(\phi'))=2s-1$.
If $l({\rm Im}(\phi'))=2(s-1)$, then $l({\rm Ker}(\phi'))=l(M_s(0,\eta))-l({\rm Im}(\phi'))=2$,
and hence $l({\rm Im}(f_1)+{\rm Ker}(\phi'))=l({\rm Im}(f_1))+l({\rm Ker}(\phi'))=2s=l(M_s(0, \eta))$.
It follows that $M_s(0, \eta)={\rm Im}(f_1)\oplus{\rm Ker}(\phi')$, a contradiction.
Therefore, $l({\rm Im}(\phi'))=2s-1=l(\O^{s-1}V(1))$, which shows that
$\phi': M_s(0, \eta)\ra \O^{s-1}V(1)$ is an epimorphism.

Note that $\phi'$ is not a split epimorphism
since $M_s(0, \eta)$ is indecomposable. From \cite[Theorem 3.5(2)]{Ch4}, there is an almost split sequence
$$0\ra\O^{s+1}V(1)\ra \O^sV(0)\oplus\O^sV(0)\xrightarrow{(f', g')}\O^{s-1}V(1)\ra 0.$$
Hence $\phi'$ factors through $(f', g')$, that is, there is a $D_4$-module map
$\left(
  \begin{array}{c}
    \s_s \\
    \s_s' \\
  \end{array}
\right): M_s(0, \eta)\ra \O^sV(0)\oplus\O^sV(0)$ such that $f'\s_s+g'\s_s'=\phi'$.
Without losing generality, we may assume $l({\rm Im}(f'\s_s))\>l({\rm Im}(g'\s_s'))$.
Then $f'\s_s\neq 0$ since $\phi'\neq 0$.
We first show that $f'\s_s$ is an epimorphism. If ${\rm Im}(g'\s_s')\subseteq{\rm soc}(\O^{s-1}V(1))$,
then $g'\s_s'$ induces a $D_4$-module map $\ol{g'\s_s'}: M_s(0, \eta)/(JM_s(0, \eta))\ra{\rm soc}(\O^{s-1}V(1))$.
Since $s-1$ is even, ${\rm soc}(\O^{s-1}V(1))\cong(s-1)V(0)$. However, $M_s(0, \eta)/(JM_s(0, \eta))\cong sV(1)$.
It follows that $\ol{g'\s_s'}=0$, and hence $g'\s_s'=0$. In this case, $f'\s_s=\phi'$ is an epimorphism.
If ${\rm Im}(g'\s_s')=\O^{s-1}V(1)$, then $f'\s_s$ is also an epimorphism by $l({\rm Im}(f'\s_s))\>l({\rm Im}(g'\s_s'))$.
Now assume that ${\rm Im}(g'\s_s')\neq\O^{s-1}V(1)$ and ${\rm Im}(g'\s_s')\nsubseteq{\rm soc}(\O^{s-1}V(1))$.
Then ${\rm rl}({\rm Im}(g'\s_s'))=2$ and $l({\rm Im}(g'\s_s'))\<2(s-1)$ by $l(\O^{s-1}V(1))=2s-1$.
Let $i=l({\rm Im}(g'\s_s')/{\rm soc}({\rm Im}(g'\s_s')))$ and $j=l({\rm soc}({\rm Im}(g'\s_s')))$.
Then $[{\rm Im}(g'\s_s')]=j[V(0)]+i[V(1)]$, and hence
$[{\rm Ker}(g'\s_s')]=[M_s(0, \eta)]-[{\rm Im}(g'\s_s')]=(s-j)[V(0)]+(s-i)[V(1)]$ in $G_0(D_4)$.
It follows that $l({\rm Ker}(g'\s_s')/{\rm soc}({\rm Ker}(g'\s_s')))=s-i$ and $l({\rm soc}({\rm Ker}(g'\s_s')))=s-j$.
By Lemma \ref{2.26}, we have $1\<i\<j\<s-1$. By Lemma \ref{2.26-2}, we have
$s-i\<s-j$, and hence $j\<i$. It follows that $1\<i=j\<s-1$ and $1\<s-i=s-j\<s-1$.
Again by Lemma \ref{2.26-2}, ${\rm Ker}(g'\s_s')$ contains an indecomposable
summand of $(t, t)$-type for some $1\<t\<s-i$. Then by \cite[Theorem 3.10(2)]{Ch4}, ${\rm Ker}(g'\s_s')$ contains
an indecomposable submodule of $(1,1)$-type. Since $f'\s_s\neq 0$, ${\rm Im}(f'\s_s)\nsubseteq{\rm soc}(\O^{s-1}V(1))$
by the same argument as above for $g'\s_s'$. If ${\rm Im}(f'\s_s)\neq\O^{s-1}V(1)$, then one can similarly check that
${\rm Ker}(f'\s_s)$ contains an indecomposable submodule of $(1,1)$-type. From \cite[Theorem 3.10(2)]{Ch4}, one knows that
$M_s(0, \eta)$ has a unique submodule of $(1, 1)$-type, denoted by $N$. Hence
$N\subseteq {\rm Ker}(f'\s_s)\cap {\rm Ker}(g'\s'_s)\subseteq {\rm Ker}(\phi')$, and so
$l({\rm Ker}(\phi'))\>l(N)=2$. Thus, $l({\rm Im}(\phi'))=l(M_s(0, \eta))-l({\rm Ker}(\phi'))\<2s-2<l(\O^{s-1}V(1))$,
which is impossible since $\phi'$ is surjective. Therefore, $f'\s_s$ is an epimorphism
from $M_s(0, \eta)$ to $\O^{s-1}V(1)$. Then by an argument similar to the one for
$\s_2$ before, one can easily check that
$\s_s: M_s(0, \eta)\ra \O^sV(0)$ is a monomorphism.

If $s$ is even, then the same argument as above shows that there is a monomorphism
$\s_s: M_s(0, \eta)\ra \O^sV(1)$.
\end{proof}

\begin{lemma}\label{2.28}
Let $s\>1$, $r\in{\mathbb Z}_2$ and $\eta\in k$. Then there is a basis $\{v_{1,1}, v_{1,2}, \cdots, v_{1,s},$
$v_{2,1}, v_{2,2}, \cdots, v_{2, s}\}$ in $M_s(r,\eta)$ such that
$$\begin{array}{lll}
a\cdot v_{1,i}=v_{2,i},& b\cdot v_{1,i}=c\cdot v_{1,i}=(-1)^{r+1}v_{1,i},& 1\<i\<s,\\
d\cdot v_{1,1}=-\eta v_{2,1},& d\cdot v_{1,i}=-v_{2,i-1}-\eta v_{2,i},& 2\<i\<s,\\
a\cdot v_{2,i}=d\cdot v_{2,i}=0,& b\cdot v_{2,i}=c\cdot v_{2,i}=(-1)^rv_{2,i},& 1\<i\<s.\\
\end{array}$$
\end{lemma}

\begin{proof}
We prove the lemma by induction on $s$. For $s=1$, it follows from Section 1.
Now let $s\>2$ and suppose that $M_i(r,\eta)$ has a desired basis for all $1\<i<s$.

Let $M=M_s(r, \eta)$. Then by \cite[Theorem 3.10(2)]{Ch4}, $M$ contains a unique submodule $N$
of $(s-1,s-1)$-type. Moreover, $N\cong M_{s-1}(r,\eta)$ and $M/N\cong M_1(r,\eta)$.
By the induction hypothesis, there is a basis $\{v_{1,1}, v_{1,2}, \cdots, v_{1,s-1},
v_{2,1}, v_{2,2}, \cdots, v_{2, s-1}\}$ in $N$ such that
$$\begin{array}{lll}
a\cdot v_{1,i}=v_{2,i},& b\cdot v_{1,i}=c\cdot v_{1,i}=(-1)^{r+1}v_{1,i},& 1\<i\<s-1,\\
d\cdot v_{1,1}=-\eta v_{2,1},& d\cdot v_{1,i}=-v_{2,i-1}-\eta v_{2,i},& 2\<i\<s-1,\\
a\cdot v_{2,i}=d\cdot v_{2,i}=0,& b\cdot v_{2,i}=c\cdot v_{2,i}=(-1)^rv_{2,i},& 1\<i\<s-1.\\
\end{array}$$
Define a subspace $L$ of $N$ by $L={\rm span}\{v_{1,i}, v_{2,i}|1\<i\<s-2\}$ for $s>2$, and $L=0$ for $s=2$. Then $L$ is obviously a submodule of $N$,
and $L\cong M_{s-2}(r,\eta)$ for $s>2$ by the induction hypothesis. It follows from \cite[Theorem 3.10(2)]{Ch4}
that $M/L\cong M_2(r, \eta)$. By the structure $M_1(r,\eta)$,
there is basis $\{x_1, x_2\}$ in $M/N$ such that
\begin{equation*}
\begin{array}{lll}
a\cdot x_1=x_2, & d\cdot x_1=-\eta x_2, & b\cdot x_1=c\cdot x_1=(-1)^{r+1}x_1,\\
a\cdot x_2=0, & d\cdot x_2=0, & b\cdot x_2=c\cdot x_2=(-1)^{r}x_2.\\
\end{array}
\end{equation*}
Let $\pi: M\ra M/N$  be the canonical epimorphism. Since $x_1\in(M/N)_{(r+1)}$ and $x_2\in(M/N)_{(r)}$,
$x_1=\pi(u_1)$ and $x_2=\pi(u_2)$ for some $u_1\in M_{(r+1)}$ and $u_2\in M_{(r)}$.
Obviously, $u_1\notin N$ and $u_2\notin N$.
Note that $a\cdot M_{(r)}=d\cdot M_{(r)}=0$, $a\cdot M_{(r+1)}\subseteq M_{(r)}$ and
$d\cdot M_{(r+1)}\subseteq M_{(r)}$. From $a\cdot x_1=x_2$, one gets $\pi(a\cdot u_1)=\pi(u_2)$.
Hence $a\cdot u_1-u_2\in N\cap M_{(r)}$, and so $a\cdot u_1=u_2+x$ for some $x\in N\cap M_{(r)}=N_{(r)}$.
By replacing $u_2$ with $u_2+x$, we may assume that $x=0$, i.e., $a\cdot u_1=u_2$.
From $d\cdot x_1=-\eta x_2$, one gets $\pi(d\cdot u_1)=\pi(-\eta u_2)$.
Hence $d\cdot u_1+\eta u_2\in N\cap M_{(r)}$, and so
$d\cdot u_1=-\eta u_2+y$ for some $y\in N\cap M_{(r)}=N_{(r)}$. Since $\{v_{2,i}|1\<i\<s-1\}$
is a basis of $N_{(r)}$, we have $y=\sum_{i=1}^{s-1}\a_iv_{2, i}$ for some
$\a_1, \a_2, \cdots, \a_{s-1}\in k$. If $\a_{s-1}=0$ then
$y\in L$. In this case, $\{\ol{v_{1, s-1}}, \ol{v_{2, s-1}}, \ol{u_1}, \ol{u_2}\}$ is a basis
of $M/L$, where $\ol{v}$ denotes the image of $v\in M$ under the canonical epimorphism
$M\ra M/L$. Moreover,
both ${\rm span}\{\ol{v_{1, s-1}}, \ol{v_{2, s-1}}\}$ and ${\rm span}\{\ol{u_1}, \ol{u_2}\}$
are submodules of $M/L$, and $M/L={\rm span}\{\ol{v_{1, s-1}}, \ol{v_{2, s-1}}\}\oplus {\rm span}\{\ol{u_1}, \ol{u_2}\}$.
This is impossible since $M/L\cong M_2(r,\eta)$ is indecomposable. Hence $\a_{s-1}\neq 0$.
Now let $v_{1,s}=-\a_{s-1}^{-1}(u_1+\sum_{i=1}^{s-2}\a_iv_{1,i+1})$ and
$v_{2,s}=-\a_{s-1}^{-1}(u_2+\sum_{i=1}^{s-2}\a_iv_{2,i+1})$. Here we regard
$\sum_{i=1}^{s-2}\a_iv_{1,i+1}=0$ and $\sum_{i=1}^{s-2}\a_iv_{2,i+1}=0$ for $s=2$.
Then $v_{1,s}\in M_{(r+1)}\backslash N$ and $v_{2,s}\in M_{(r)}\backslash N$. Hence
$\{v_{1,1}, v_{1,2}, \cdots, v_{1,s-1}, v_{1, s},
v_{2,1}, v_{2,2}, \cdots, v_{2, s-1}, v_{2,s}\}$ is a basis of $M$. Moreover, we have
$$\begin{array}{rcl}
a\cdot v_{1,s}&=&-\a_{s-1}^{-1}(a\cdot u_1+\sum_{i=1}^{s-2}\a_ia\cdot v_{1,i+1})\\
&=&-\a_{s-1}^{-1}(u_2+\sum_{i=1}^{s-2}\a_iv_{2,i+1})\\
&=& v_{2,s}\\
\end{array}$$
and
$$\begin{array}{rcl}
d\cdot v_{1,s}&=&-\a_{s-1}^{-1}(d\cdot u_1+\sum_{i=1}^{s-2}\a_id\cdot v_{1,i+1})\\
&=&-\a_{s-1}^{-1}(-\eta u_2+\sum_{i=1}^{s-1}\a_iv_{2, i}+\sum_{i=1}^{s-2}\a_i(-v_{2,i}-\eta v_{2,i+1}))\\
&=&-\a_{s-1}^{-1}(-\eta u_2+\a_{s-1}v_{2,s-1}-\sum_{i=1}^{s-2}\a_i\eta v_{2,i+1})\\
&=& -v_{2,s-1}-\eta(-\a_{s-1}^{-1}(u_2+\sum_{i=1}^{s-2}\a_iv_{2,i+1}))\\
&=& -v_{2,s-1}-\eta v_{2,s}.\\
\end{array}$$
This shows that $\{v_{1,1}, v_{1,2}, \cdots, v_{1,s-1}, v_{1, s},
v_{2,1}, v_{2,2}, \cdots, v_{2, s-1}, v_{2,s}\}$ is a desired basis of $M$.
\end{proof}

\begin{lemma}\label{2.29}
Let $s\>1$ and $r\in{\mathbb Z}_2$. Then there is a basis $\{v_{1,1}, v_{1,2}, \cdots, v_{1,s},
v_{2,1},$ $v_{2,2}, \cdots, v_{2, s}\}$ in $M_s(r,\infty)$ such that
$$\begin{array}{lll}
a\cdot v_{1,1}=0,& a\cdot v_{1,i}=v_{2,i-1},& 2\<i\<s,\\
b\cdot v_{1,i}=c\cdot v_{1,i}=(-1)^{r+1}v_{1,i},& d\cdot v_{1,i}=v_{2,i},& 1\<i\<s,\\
a\cdot v_{2,i}=d\cdot v_{2,i}=0,& b\cdot v_{2,i}=c\cdot v_{2,i}=(-1)^rv_{2,i},& 1\<i\<s.\\
\end{array}$$
\end{lemma}

\begin{proof}
It is similar to Lemma \ref{2.28}.
\end{proof}

\begin{lemma}\label{2.30}
Let $\eta\in{\mathbb P}^1(k)$ and $r, r'\in{\mathbb Z}_2$.
Then $M_1(r,\eta)\ot M_1(r', \eta)\cong M_1(0, \eta)\oplus M_1(1, \eta)$.
\end{lemma}

\begin{proof}
By Proposition \ref{2.1}, it is enough to show that
$M_1(0,\eta)\ot M_1(0, \eta)\cong M_1(0, \eta)\oplus M_1(1, \eta)$
for all $\eta\in{\mathbb P}^1(k)$.

Assume $\eta\in k$. Let $\{v_1, v_2\}$ be the standard basis of $M_1(0, \eta)$ as stated in Section 1
(or in Lemma \ref{2.28}).
Let $u_1=v_1\ot v_1$, $u_2=v_1\ot v_2-v_2\ot v_1$, $w_1=v_1\ot v_2$ and $w_2=v_2\ot v_2$ in
$M_1(0,\eta)\ot M_1(0, \eta)$. Then $\{u_1, u_2, w_1, w_2\}$ is a basis of
$M_1(0,\eta)\ot M_1(0, \eta)$. Putting $U={\rm span}\{u_1, u_2\}$ and $W={\rm span}\{w_1, w_2\}$.
Then by a straightforward verification, one can show that both $U$ and $W$ are submodules of
$M_1(0,\eta)\ot M_1(0, \eta)$, and that $U\cong M_1(1, \eta)$ and $W\cong M_1(0, \eta)$.
It follows that $M_1(0,\eta)\ot M_1(0, \eta)=U\oplus W\cong M_1(1,\eta)\oplus M_1(0, \eta)$.

Similarly, one can show that $M_1(0,\infty)\ot M_1(0, \infty)\cong M_1(1,\infty)\oplus M_1(0, \infty)$.
\end{proof}

\begin{lemma}\label{2.31}
Let $s\>1$, $\eta\in{\mathbb P}^1(k)$ and $r\in{\mathbb Z}_2$.
Then $M_s(r,\eta)\ot M_s(r, \eta)$ contains a submodule isomorphic to $M_s(1,\eta)$.
\end{lemma}

\begin{proof}
It follows from Lemma \ref{2.30} that $M_1(r,\eta)\ot M_1(r, \eta)$ contains a submodule isomorphic to $M_1(1,\eta)$.
Now let $s\>2$. By Propositions \ref{2.1} and \ref{2.4}, we have
$M_s(r,\eta)\ot M_s(r, \eta)\cong V(r)\ot M_s(0, \eta)\ot V(r)\ot M_s(0, \eta)\cong M_s(0,\eta)\ot M_s(0, \eta)$.
Hence we only need to show that $M_s(0,\eta)\ot M_s(0, \eta)$ contains a submodule isomorphic to $M_s(1,\eta)$.

Let $\{v_{1,1}, v_{1,2}, \cdots, v_{1,s}, v_{2,1}, v_{2,2}, \cdots, v_{2, s}\}$ be the basis of $M_s(0,\infty)$
as stated in Lemma \ref{2.29}. Then $\{v_{i, j}\ot v_{m,n}|1\<i,m\<2, 1\<j, n\<s\}$ is a basis of
$M_s(0,\infty)\ot M_s(0,\infty)$. For any $1\<i\<s$, let $u_{1,i}=\sum_{j=1}^iv_{1,j}\ot v_{1, i+1-j}$
and $u_{2,i}=\sum_{j=1}^i(v_{1,j}\ot v_{2, i+1-j}-v_{2,j}\ot v_{1,i+1-j})$. Then obviously,
$\{u_{1,i}, u_{2,i}|1\<i\<s\}$ is a linearly independent subset of $M_s(0,\infty)\ot M_s(0,\infty)$,
and $b\cdot u_{1,i}=c\cdot u_{1,i}=u_{1,i}$ and $b\cdot u_{2,i}=c\cdot u_{2,i}=-u_{2,i}$ for all $1\<i\<s$.
Now we have $a\cdot u_{1,1}=a\cdot(v_{1,1}\ot v_{1,1})=0$ and for $2\<i\<s$
$$\begin{array}{rcl}
a\cdot u_{1,i}&=&\sum\limits_{j=1}^ia\cdot(v_{1,j}\ot v_{1, i+1-j})\\
&=&\sum\limits_{j=1}^i(v_{1,j}\ot a\cdot v_{1, i+1-j}+a\cdot v_{1,j}\ot b\cdot v_{1, i+1-j})\\
&=&\sum\limits_{j=1}^{i-1}v_{1,j}\ot v_{2, i-j}-\sum\limits_{j=2}^iv_{2,j-1}\ot v_{1, i+1-j}\\
&=&\sum\limits_{j=1}^{i-1}(v_{1,j}\ot v_{2, i-j}-v_{2,j}\ot v_{1, i-j})=u_{2,i-1}.\\
\end{array}$$
Furthermore, for any $1\<i\<s$, we have
$$\begin{array}{rcl}
d\cdot u_{1,i}&=&\sum\limits_{j=1}^id\cdot(v_{1,j}\ot v_{1, i+1-j})\\
&=&\sum\limits_{j=1}^i(v_{1,j}\ot d\cdot v_{1, i+1-j}+d\cdot v_{1,j}\ot c\cdot v_{1, i+1-j})\\
&=&\sum\limits_{j=1}^i(v_{1,j}\ot v_{2, i+1-j}-v_{2,j}\ot v_{1, i+1-j})\\
&=&u_{2,i},\\
\end{array}$$
$$\begin{array}{rcl}
a\cdot u_{2,i}&=&\sum\limits_{j=1}^ia\cdot(v_{1,j}\ot v_{2, i+1-j}-v_{2,j}\ot v_{1,i+1-j})\\
&=&\sum\limits_{j=1}^i(a\cdot v_{1,j}\ot b\cdot v_{2, i+1-j}-v_{2,j}\ot a\cdot v_{1,i+1-j})\\
&=&\sum\limits_{1<j\<i} v_{2,j-1}\ot v_{2, i+1-j}-\sum\limits_{1\<j<i}v_{2,j}\ot v_{2,i-j}\\
&=&0,\\
\end{array}$$
and similarly $d\cdot u_{2,i}=0$.
Therefore, ${\rm span}\{u_{1,i}, u_{2,i}|1\<i\<s\}$ is a submodule of $M_s(0,\infty)\ot M_s(0,\infty)$.
It follows from Lemma \ref{2.29} that ${\rm span}\{u_{1,i}, u_{2,i}|1\<i\<s\}$ is isomorphic to $M_s(1,\infty)$.

Now let $\eta\in k$ and let $\{v_{1,1}, v_{1,2}, \cdots, v_{1,s}, v_{2,1}, v_{2,2}, \cdots, v_{2, s}\}$
be the basis of $M_s(0,\eta)$ as stated in Lemma \ref{2.28}. Then $\{v_{i, j}\ot v_{m,n}|1\<i,m\<2, 1\<j, n\<s\}$
is a basis of $M_s(0,\eta)\ot M_s(0,\eta)$. For any $1\<i\<s$, let $u_{1,i}=\sum_{j=1}^iv_{1,j}\ot v_{1, i+1-j}$
and $u_{2,i}=\sum_{j=1}^i(v_{1,j}\ot v_{2, i+1-j}-v_{2,j}\ot v_{1,i+1-j})$. Then by a similar argument as above,
one can show that ${\rm span}\{u_{1,i}, u_{2,i}|1\<i\<s\}$ is a submodule of $M_s(0,\eta)\ot M_s(0,\eta)$,
and is isomorphic to $M_s(1,\eta)$ by Lemma \ref{2.28}.
\end{proof}

\begin{proposition}\label{2.32}
Let $t\>s\>1$, $r, r'\in{\mathbb Z}_2$ and $\eta\in{\mathbb P}^1(k)$. Then
$$M_s(r,\eta)\ot M_t(r', \eta)\cong s(t-1)P(r+r')\oplus M_s(0, \eta)\oplus M_s(1, \eta).$$
\end{proposition}

\begin{proof}
By Propositions \ref{2.1} and \ref{2.4}, it is enough to show the proposition for $r=r'=0$.
We only consider the case that $t$ is odd since the proof is similar for the other case.

Assume that $t$ is odd. Then by Lemma \ref{2.27}, there is an exact sequence
$$0\ra M_t(0, \eta)\ra \O^tV(0)\ra V(1)\ra 0.$$
Applying $M_s(0, \eta)\ot$ to the above sequence, one gets the following exact sequence
$$0\ra M_s(0, \eta)\ot M_t(0, \eta)\xrightarrow{\s} M_s(0, \eta)\ot\O^tV(0)\ra M_s(0, \eta)\ot V(1)\ra 0.$$
By \cite[Theorem 3.10(2)]{Ch4}, $M_t(0, \eta)$ contains a unique submodule $M$ of $(s, s)$-type, and
$M\cong M_s(0, \eta)$. From Lemma \ref{2.31}, one knows that $M_s(0, \eta)\ot M\cong M_s(0, \eta)\ot M_s(0, \eta)$
contains a submodule isomorphic to $M_s(1,\eta)$. It follows that $M_s(0, \eta)\ot M_t(0, \eta)$ contains a submodule $N$
such that $N\cong M_s(1, \eta)$. From Proposition \ref{2.21},
$M_s(0, \eta)\ot\O^tV(0)$ contains submodules $P$ and $M'$ with $P\cong stP(0)$ and $M'\cong M_s(1, \eta)$
such that $M_s(0, \eta)\ot\O^tV(0)=P\oplus M'$. Since $\s$ is a monomorphism, $\s(N)\cong N\cong M_s(1, \eta)$,
and hence ${\rm soc}(\s(N))\cong sV(1)$. However, ${\rm soc}(P)\cong stV(0)$ since ${\rm soc}(P(0))\cong V(0)$.
It follows that the sum $P+\s(N)$ is direct, and so $M_s(0, \eta)\ot\O^tV(0)=P\oplus M'=P\oplus \s(N)$ by
comparing their lengths. By Proposition \ref{2.4}, we have $M_s(0, \eta)\ot V(1)\cong M_s(1, \eta)$. Hence
we have the following exact sequence
$$0\ra M_s(0, \eta)\ot M_t(0, \eta)\xrightarrow{\s} P\oplus \s(N)\xrightarrow{f}M_s(1, \eta)\ra 0.$$
Since $f$ is an epimorphism and $f(\s(N))=0$, $f|_P: P\ra M_s(1, \eta)$ is an epimorphism.
It follows that $M_s(0, \eta)\ot M_t(0, \eta)\cong {\rm Ker}(f)={\rm Ker}(f|_P)\oplus\s(N)\cong
s(t-1)P(0)\oplus \O M_s(1, \eta)\oplus M_s(1, \eta)\cong
s(t-1)P(0)\oplus M_s(0, \eta)\oplus M_s(1, \eta)$.
\end{proof}

\section{\bf Generators and relations for the Green ring of $D_4$}

In this section, we will consider the Green ring $r(D_4)$ of $D_4$.
At first, $r(D_4)$ is a commutative ring. Moreover,
the duality $(-)^*$ of mod$D_4$ induces a ring involution of $r(D_4)$
determined by $[M]^*=[M^*]$ for any $M\in{\rm mod}D_4$, as stated in Section 1. That is,
$r(D_4)\ra r(D_4)$, $x\mapsto x^*$ is a ring automorphism of $r(D_4)$ and
$x^{**}=x$ for all $x\in r(D_4)$.

Let $g=[V(1)]$, $x=[V(2,0)]$, $y=[\O V(0)]$ and $z=[\O^{-1}V(0)]$ in $r(D_4)$.
Then we have the following lemma.

\begin{lemma}\label{3.1}
The following relations are satisfied in $r(D_4)$:\\
$(1)$ $g^2=1$;\\
$(2)$ $x^2=[P(1)]$ and $x^3=2x+2gx$;\\
$(3)$ $xy=xz=x+2gx=x(1+2g)$;\\
$(4)$ $yz=1+2x^2$.
\end{lemma}

\begin{proof}
Part (1) follows from Proposition \ref{2.1} since $[V(0)]=1$ in $r(D_4)$.
By Proposition \ref{2.1}, $gx=[V(2,1)]$. Part (2) follows from Proposition \ref{2.5} and Corollary \ref{2.8}.
Part (3) follows from Corollary \ref{2.7}(1). Part (4) follows from Part (2) and Corollary \ref{2.19}(1).
\end{proof}

\begin{lemma}\label{3.2}
For all $n\>1$, define $a_n\in\mathbb Z$ by $a_n=\frac{1}{2}\sum_{i=1}^{n-1}(3^{i-1}+1)(n-i)$
for $n>1$ and $a_1=0$. Then $3a_n-\frac{n(n-1)}{2}=a_{n+1}-n$ for all $n\>1$.
\end{lemma}

\begin{proof}
For $n=1$, $a_2-1=1-1=0=3a_1-\frac{1(1-1)}{2}$. Now let $n>1$. Then
$$\begin{array}{rcl}
3a_n-\frac{n(n-1)}{2}&=&\frac{1}{2}\sum_{i=1}^{n-1}(3^i+3)(n-i)-\frac{n(n-1)}{2}\\
&=&\frac{1}{2}\sum_{i=1}^{n-1}(3^i+1)(n-i)+\sum_{i=1}^{n-1}(n-i)-\frac{n(n-1)}{2}\\
&=&\frac{1}{2}\sum_{i=2}^{n}(3^{i-1}+1)(n-(i-1))\\
&=&\frac{1}{2}\sum_{i=1}^{n}(3^{i-1}+1)(n+1-i)-n\\
&=&a_{n+1}-n.\\
\end{array}$$
\end{proof}

For all $n\>1$, define $f_n\in r(D_4)$ by $f_n=a_n(1+g)-\frac{n(n-1)}{2}g^n$,
where $a_n$ are given as in Lemma \ref{3.2}.

\begin{lemma}\label{3.3}
$[\O^nV(0)]=y^n-f_nx^2$ in $r(D_4)$ for all $n\>1$.
\end{lemma}

\begin{proof}
We prove the lemma by induction on $n$. For $n=1$, since $f_1=0$, we have
$[\O V(0)]=y=y-f_1x^2$. Now assume $n\>1$. Then by the induction hypothesis,
Lemmas \ref{3.1} and \ref{3.2}, we have
$$\begin{array}{rcl}
[\O^nV(0)\otimes \O V(0)]&=&[\O^{n}V(0)][\O V(0)]\\
&=&(y^{n}-f_{n}x^2)y\\
&=&y^{n+1}-f_{n}x^2y\\
&=&y^{n+1}-(a_n(1+g)-\frac{n(n-1)}{2}g^n)(1+2g)x^2\\
&=&y^{n+1}-(a_n(3+3g)-\frac{n(n-1)}{2}(g^n+2g^{n+1}))x^2\\
&=&y^{n+1}-((3a_n-\frac{n(n-1)}{2})(1+g)-\frac{n(n-1)}{2}g^{n+1})x^2\\
&=&y^{n+1}-((a_{n+1}-n)(1+g)-\frac{n(n-1)}{2}g^{n+1})x^2.\\
\end{array}$$
On the other hand, by Proposition \ref{2.1} and Lemma \ref{3.1}, we have $gx^2=[P(0)]$.
Thus, from Lemma \ref{2.13}, one gets that
$[\O^nV(0)\otimes \O V(0)]=[\O^{n+1}V(0)]+ng^nx^2$.
Hence we have
$$\begin{array}{rcl}
[\O^{n+1}V(0)]&=&y^{n+1}-((a_{n+1}-n)(1+g)-\frac{n(n-1)}{2}g^{n+1})x^2-ng^nx^2\\
&=&y^{n+1}-(a_{n+1}(1+g)-n(g^n+g^{n+1})-\frac{n(n-1)}{2}g^{n+1}+ng^n)x^2\\
&=&y^{n+1}-(a_{n+1}(1+g)-\frac{n(n+1)}{2}g^{n+1})x^2\\
&=&y^{n+1}-f_{n+1}x^2.\\
\end{array}$$
\end{proof}

\begin{corollary}\label{3.4}
$[\O^{-n}V(0)]=z^n-f_nx^2$ in $r(D_4)$ for all $n\>1$.
\end{corollary}

\begin{proof}
By Proposition \ref{2.1} and Lemma \ref{2.15}, we have $g^*=g$ and $x^*=gx$, and so $f_n^*=f_n$.
By Lemma \ref{2.16}, one gets that $y^*=z$ and $[\O^nV(0)]^*=[\O^{-n}V(0)]$ for all $n\>1$.
It follows from Lemma \ref{3.3} that
$[\O^{-n}V(0)]=[\O^nV(0)]^*=(y^n-f_nx^2)^*=y^{*n}-f_n^*x^{*2}=z^n-f_n(gx)^2
=z^n-f_nx^2$.
\end{proof}

Let $R$ be the subring of $r(D_4)$ generated by $g$, $x$, $y$ and $z$.
Then we have the following proposition.

\begin{proposition}\label{3.5}
The subring $R$ of $r(D_4)$ is generated, as a $\mathbb Z$-module, by the following set:
$$\{[V(r)], [V(2,r)], [P(r)], [\O^nV(r)], [\O^{-n}V(r)]|r\in{\mathbb Z}_2, n\>1\}.$$
Consequently, $R$ is a free abelian group with the above set as a $\mathbb Z$-basis.
\end{proposition}

\begin{proof}
Let $R'$ be the $\mathbb Z$-submodule of $r(D_4)$ generated by the set given in the proposition.
Then $R'$ is obviously a free $\mathbb Z$-module with the set given above as a $\mathbb Z$-basis.
From the discussion in the last section, one can see that
$R'$ is closed with respect to the multiplication of $r(D_4)$.
Note that $1=[V(0)]\in R'$. It follows that $R'$ is a subring of $r(D_4)$.
Hence $R\subseteq R'$ by $g, x, y, z\in R'$.

Conversely, we first have that $[V(1)], [V(2,0)]\in R$. Since $R$ is a subring of $r(D_4)$, $[V(0)]=1\in R$ and
$[V(2,1)]=[V(1)\ot V(2, 0)]=gx\in R$. By Lemma \ref{3.1}, $[P(1)]=x^2\in R$, and hence
$[P(0)]=[V(1)\ot P(1)]=gx^2\in R$. From Lemma \ref{3.3} and Corollary \ref{3.4},
one gets that $[\O^nV(0)]\in R$ and $[\O^{-n}V(0)]\in R$ for all $n\>1$. Then by Corollary \ref{2.3},
we have $[\O^nV(1)]=[V(1)\ot \O^nV(0)]=g[\O^nV(0)]\in R$, and similarly
$[\O^{-n}V(1)]=g[\O^{-n}V(0)]\in R$ for all $n\>1$. Therefore, $R'\subseteq R$.
\end{proof}

\begin{corollary}\label{3.6}
The following set is a $\mathbb Z$-basis of $R$:
$$\{1,\ g,\ x,\ gx,\ x^2,\ gx^2,\ y^n,\ gy^n,\ z^n,\ gz^n|n\>1\}.$$
\end{corollary}

\begin{proof}
Let $R_1$ be the subring of $r(D_4)$ generated by $g$ and $x$.
From Lemma \ref{3.1} and the proof of Proposition \ref{3.5}, it follows that $R_1$ is a free $\mathbb Z$-module
with a $\mathbb Z$-basis $\{1, g, x, gx, x^2, gx^2\}$. By Lemma \ref{3.3}, Corollary \ref{3.4} and
the proof of Proposition \ref{3.5}, we have that $[\O^nV(1)]=g[\O^nV(0)]=gy^n-gf_nx^2$ and
$[\O^{-n}V(1)]=g[\O^{-n}V(0)]=gz^n-gf_nx^2$ for all $n\>1$.
Note that $f_nx^2, gf_nx^2\in R_1$ for all $n\>1$. Consider the canonical
$\mathbb Z$-module epimorphism $\pi: R\ra R/R_1$. Then from Lemma \ref{3.3}, Corollary \ref{3.4} and
the discussion above, we have $\pi([\O^nV(r)])=\pi(g^ry^n)$ and $\pi([\O^{-n}V(r)])=\pi(g^rz^n)$
for all $n\>1$ and $r\in{\mathbb Z}_2$. Thus, the corollary follows from Proposition \ref{3.5}.
\end{proof}

Let ${\mathbb Z}[g_1, x_1, y_1, z_1]$ be the polynomial algebra over $\mathbb Z$
in four variables $g_1, x_1, y_1, z_1$. Let $I$ be the ideal of ${\mathbb Z}[g_1, x_1, y_1, z_1]$
generated by the following elements:
$$g_1^2-1,\ x_1^3-2x_1(1+g_1),\ x_1(y_1-1-2g_1),\ x_1(y_1-z_1),\ y_1z_1-1-2x_1^2.$$

\begin{theorem}\label{3.7}
The subring $R$ of the Green ring $r(D_4)$ is isomorphic to the quotient ring
${\mathbb Z}[g_1, x_1, y_1, z_1]/I$.
\end{theorem}

\begin{proof}
Since $R$ is a commutative ring, there is a unique ring homomorphism
$\phi: {\mathbb Z}[g_1, x_1, y_1, z_1]\ra R$ such that $\phi(g_1)=g$, $\phi(x_1)=x$,
$\phi(y_1)=y$ and $\phi(z_1)=z$. Since $R$ is generated by $\{g, x, y, z\}$ as a ring,
$\phi$ is an epimorphism. From Lemma \ref{3.1}, one can easily check that
$\phi(g_1^2-1)=0$, $\phi(x_1^3-2x_1(1+g_1))=0$, $\phi(x_1(y_1-1-2g_1))=0$,
$\phi(x_1(y_1-z_1))=0$ and $\phi(y_1z_1-1-2x_1^2)=0$. It follows that $\phi(I)=0$.
Hence $\phi$ induces a ring epimorphism $\ol{\phi}: {\mathbb Z}[g_1, x_1, y_1, z_1]/I\ra R$
such that $\ol{\phi}(\ol{u})=\phi(u)$ for all $u\in{\mathbb Z}[g_1, x_1, y_1, z_1]$,
where $\ol{u}$ denotes the image of $u$ under the canonical epimorphism
${\mathbb Z}[g_1, x_1, y_1, z_1]\ra {\mathbb Z}[g_1, x_1, y_1, z_1]/I$.
By Corollary \ref{3.6}, one can define a $\mathbb Z$-module map $\psi: R\ra {\mathbb Z}[g_1, x_1, y_1, z_1]/I$
by
$$\begin{array}{lllll}
\psi(1)=\ol{1}, &\psi(g)=\ol{g_1}, &\psi(x)=\ol{x_1}, &\psi(gx)=\ol{g_1x_1}, &\psi(x^2)=\ol{x_1^2},\\
\psi(gx^2)=\ol{g_1x_1^2}, &\psi(y^n)=\ol{y_1^n}, &\psi(gy^n)=\ol{g_1y_1^n}, &\psi(z^n)=\ol{z_1^n},
&\psi(gz^n)=\ol{g_1z_1^n},\\
\end{array}$$
where $n\>1$.
From the definition of $I$, one can see that ${\mathbb Z}[g_1, x_1, y_1, z_1]/I$ is
generated, as a $\mathbb Z$-module, by the following set
$$\{\ol{1},\ \ol{g_1},\ \ol{x_1},\ \ol{g_1x_1},\ \ol{x_1^2},\ \ol{g_1x_1^2},\ \ol{y_1^n},\
\ol{g_1y_1^n},\ \ol{z_1^n},\ \ol{g_1z_1^n}|n\>1\}.$$
Let $\ol{u}$ be any element in the above set. Then it is straightforward to check that
$\psi\ol{\phi}(\ol{u})=\ol{u}$. Hence $\psi\ol{\phi}={\rm id}$, which implies that $\ol{\phi}$
is a monomorphism, and so it is a ring isomorphism.
\end{proof}

Now let $X_{n,\eta}=[M_n(0,\eta)]$ in $r(D_4)$ for all $n\>1$ and $\eta\in{\mathbb P}^1(k)$.
Then we have the following lemma.

\begin{lemma}\label{3.8}
Let $n, s\>1$ and $\eta, \a\in{\mathbb P}^1(k)$. Then we have the following relations in $r(D_4)$:\\
$(1)$ $gX_{n,\eta}=[M_n(1,\eta)]$.\\
$(2)$ $xX_{n,\eta}=n(1+g)x$.\\
$(3)$ $yX_{n,\eta}=ngx^2+gX_{n,\eta}$.\\
$(4)$ $zX_{n,\eta}=nx^2+gX_{n,\eta}$.\\
$(5)$ If $\eta\neq \a$ then $X_{n,\eta}X_{s,\a}=nsgx^2$.\\
$(6)$ If $s\>n$ then $X_{n,\eta}X_{s,\eta}=n(s-1)gx^2+X_{n,\eta}+gX_{n,\eta}$.\\
\end{lemma}

\begin{proof}
We have already known that $[V(2,1)]=gx$, $[P(1)]=x^2$ and $[P(0)]=gx^2$. Then
Part (1) follows from Proposition \ref{2.4}. Part (2) follows from Corollary \ref{2.7}(3).
Part (3) follows from Part (1) and Proposition \ref{2.21}(1).
Part (4) follows from Part (1) and Corollary \ref{2.24}(1).
Part (5) follows from Proposition \ref{2.25}. Part (6) follows from Proposition \ref{2.32} and Part (1).
\end{proof}

Let ${\mathbb Z}[X]$ be the polynomial algebra over $\mathbb Z$ in the following variables:
$$X=\{g_1, x_1, y_1, z_1, X'_{n,\eta}|n\>1, \eta\in{\mathbb P}^1(k)\}.$$
Let $J$ be the ideal of ${\mathbb Z}[X]$ generated by the following subset
$$G=\left\{\left.\begin{array}{l}
g_1^2-1,\ x_1^3-2x_1(1+g_1),\ x_1(y_1-1-2g_1),\\
x_1(y_1-z_1),\ y_1z_1-1-2x_1^2,\\
x_1X'_{n,\eta}-n(1+g_1)x_1,\ y_1X'_{n,\eta}-ng_1x_1^2-g_1X'_{n,\eta},\\
z_1X'_{n,\eta}-nx_1^2-g_1X'_{n,\eta},\ X'_{n,\eta}X'_{s,\a}-nsg_1x_1^2,\\
X'_{n,\eta}X'_{t,\eta}-n(t-1)g_1x_1^2-X'_{n,\eta}-g_1X'_{n,\eta}\\
\end{array}\right|
\begin{array}{l}
n, s, t\>1\\
\mbox{with }t\>n,\\
\eta, \a\in{\mathbb P}^1(k)\\
\mbox{with }\eta\neq\a\\
\end{array}\right\}.$$

\begin{theorem}\label{3.9}
The Green ring $r(D_4)$ of $D_4$ is isomorphic to the quotient ring
${\mathbb Z}[X]/J$.
\end{theorem}

\begin{proof}
Since $r(D_4)$ is a commutative ring, there is a unique ring homomorphism
$f: {\mathbb Z}[X]\ra r(D_4)$ such that
$$f(g_1)=g,\ f(x_1)=x,\ f(y_1)=y,\ f(z_1)=z,\ f(X'_{n,\eta})=X_{n,\eta}$$
for all $n\>1$ and $\eta\in{\mathbb P}(k)$. By Lemma \ref{3.8}(1),
$[M_n(1,\eta)]=gX_{n,\eta}$ for $n\>1$ and $\eta\in{\mathbb P}(k)$.
It follows from Proposition \ref{3.5} that $r(D_4)$ is generated, as a ring, by
$\{g, x, y, z, X_{n,\eta}|n\>1, \eta\in{\mathbb P}^1(k)\}$, which implies
that $f$ is an epimorphism. By Lemmas \ref{3.1} and \ref{3.8},
it is straightforward to check that $f(u)=0$ for all $u\in G$. Hence
$f(J)=0$, and so $f$ induces a unique ring epimorphism
$\ol{f}: {\mathbb Z}[X]/J\ra r(D_4)$ such that $\ol{f}(\ol{u})=f(u)$
for all $u\in{\mathbb Z}[X]$, where $\ol{u}$ denotes the image of $u$
under the canonical epimorphism ${\mathbb Z}[X]\ra{\mathbb Z}[X]/J$.
Note that ${\mathbb Z}[g_1, x_1, y_1, z_1]$ is a subring of ${\mathbb Z}[X]$
since $\{g_1, x_1, y_1, z_1\}\subset X$. Obviously,
$f({\mathbb Z}[g_1, x_1, y_1, z_1])=R$ and
$I\subseteq J\cap{\mathbb Z}[g_1, x_1, y_1, z_1]$, where $R$ and $I$ are given
as before. Therefore, there is a ring homomorphism
$\tau: {\mathbb Z}[g_1, x_1, y_1, z_1]/I\ra{\mathbb Z}[X]/J$ given by
$\tau(u+I)=\ol{u}$ for all $u\in{\mathbb Z}[g_1, x_1, y_1, z_1]$.
Consider the composition of ring homomorphisms
$$\theta: {\mathbb Z}[g_1, x_1, y_1, z_1]/I\xrightarrow{\tau}{\mathbb Z}[X]/J
\xrightarrow{\ol{f}}r(D_4).$$
Then ${\rm Im}(\theta)=R$. Hence $\theta$
can be regarded as a ring homomorphism
$$\theta: {\mathbb Z}[g_1, x_1, y_1, z_1]/I\ra R.$$
One can easily see that $\theta$ is exactly the ring
isomorphism $\ol{\phi}: {\mathbb Z}[g_1, x_1, y_1, z_1]/I\ra R$ described in
the proof of Theorem \ref{3.7}. Hence $\theta$ is injective, and so is $\tau$,
which implies that $I=J\cap{\mathbb Z}[g_1, x_1, y_1, z_1]$.
Moreover, $\ol{f}|_{{\rm Im}(\tau)}: {\rm Im}(\tau)\ra R$ is a ring isomorphism.
Let $(\ol{f}|_{{\rm Im}(\tau)})^{-1}: R\ra{\rm Im}(\tau)$ be the inverse.

Let $R_0$ be the $\mathbb Z$-submodule of $r(D_4)$ generated by
$\{X_{n,\eta}, gX_{n,\eta}|n\>1, \eta\in{\mathbb P}^1(k)\}$.
Then $R_0$ is a free $\mathbb Z$-module
with the basis $\{X_{n,\eta}, gX_{n,\eta}|n\>1, \eta\in{\mathbb P}^1(k)\}$.
It follows from Proposition \ref{3.5} that $r(D_4)=R\oplus R_0$ as $\mathbb Z$-modules.
Hence one can define a $\mathbb Z$-module homomorphism
$\psi: r(D_4)\ra {\mathbb Z}[X]/J$ by $\psi(v)=(\ol{f}|_{{\rm Im}(\tau)})^{-1}(v)$
for all $v\in R$, $\psi(X_{n,\eta})=\ol{X'_{n,\eta}}$ and $\psi(gX_{n,\eta})=\ol{g_1X'_{n,\eta}}$
for all $n\>1$ and
$\eta\in{\mathbb P}^1(k)$. By the definition of $\tau$, one can see that
${\rm Im}(\tau)$ is generated, as a subring of ${\mathbb Z}[X]/J$, by
$\{\ol{g_1}, \ol{x_1}, \ol{y_1}, \ol{z_1}\}$. Then from the definition of $J$,
one gets that ${\mathbb Z}[X]/J$ is generated, as a $\mathbb Z$-module, by
${\rm Im}(\tau)\cup\{\ol{X'_{n,\eta}}, \ol{g_1X'_{n,\eta}}|n\>1, \eta\in{\mathbb P}^1(k)\}$.
Obviously, $(\psi\ol{f})|_{{\rm Im}(\tau)}={\rm id}_{{\rm Im}(\tau)}$.
For all $n\>1$ and $\eta\in{\mathbb P}^1(k)$, we have that
$(\psi\ol{f})(\ol{X'_{n,\eta}})=\psi(f(X'_{n,\eta}))=\psi(X_{n,\eta})=\ol{X'_{n,\eta}}$
and $(\psi\ol{f})(\ol{g_1X'_{n,\eta}})=\psi(f(g_1X'_{n,\eta}))=\psi(gX_{n,\eta})=\ol{g_1X'_{n,\eta}}$.
It follows that $\psi\ol{f}$ is the identity map on ${\mathbb Z}[X]/J$.
Hence $\ol{f}$ is a monomorphism, and so it is a ring isomorphism.
\end{proof}

\begin{remark}
From Lemma \ref{3.8} and Theorem \ref{3.9}, one knows that $r(D_4)$ is not finitely generated
as a ring.
\end{remark}

\centerline{ACKNOWLEDGMENTS}

This work is supported by NSF of China (No. 11171291).\\

\end{document}